\documentclass[final,onefignum,onetabnum]{siamart220329}

\usepackage{lipsum}
\usepackage{amsfonts}
\usepackage{graphicx}
\usepackage{epstopdf}
\usepackage{algorithmic}

\ifpdf
  \DeclareGraphicsExtensions{.eps,.pdf,.png,.jpg}
\else
  \DeclareGraphicsExtensions{.eps}
\fi

\usepackage{enumitem}

\newsiamremark{remark}{Remark}

\headers{Quantify and accelerate training in NN-based PDE solvers}{C. Chen, Q. Zhou, Y. Yang,Y. Xiang and T. Luo}

\title{Quantifying Training Difficulty and Accelerating Convergence in Neural Network-Based PDE Solvers\thanks{Corresponding authors
\funding{
This work of T.L. is sponsored by the National Key R\&D Program of China Grant No. 2022YFA1008200 (T. L.), the National Natural Science Foundation of China Grant No. 12101401 (T. L.),  Shanghai Municipal Science and Technology Key Project No. 22JC1401500 (T. L.). The work of Y.X. was supported by the Project of Hetao Shenzhen-HKUST Innovation Cooperation Zone HZQB-KCZYB-2020083.}}}

\author{
Chuqi Chen\thanks{Department of Mathematics, Hong Kong University of Science and Technology,
	Clear Water Bay, Hong Kong SAR, China} (\email{cchenck@connect.ust.hk})
\and Qixuan Zhou\thanks{School of Mathematical Sciences, Shanghai Jiao Tong University, Shanghai, China}  (\email{zhouqixuan@sjtu.edu.cn})
\and Yahong Yang\thanks{Department of Mathematics,
	The Pennsylvania State University, University Park, State College, PA, USA} (\email{yxy5498@psu.edu})
 \and Yang Xiang\footnotemark[2],\thanks{Algorithms of Machine Learning and Autonomous Driving Research Lab,
    HKUST Shenzhen-Hong Kong Collaborative Innovation Research Institute,
    Futian, Shenzhen, China}(\email{maxiang@ust.hk})
    \and Tao Luo$^{*}$\thanks{School of Mathematical Sciences, Institute of Natural Sciences, MOE-LSC
Shanghai Jiao Tong University, CMA-Shanghai, Shanghai Artificial Intelligence Laboratory, Shanghai, China}(\email{luotao41@sjtu.edu.cn})}

\ifpdf
\hypersetup{
  pdftitle={Quantify and accelerate training in NN-based PDE solvers},
  pdfauthor={Chuqi Chen, Qixuan Zhou, Yahong Yang, Yang Xiang and Tao Luo}
}
\fi

\usepackage[T1]{fontenc}
\usepackage[latin1]{inputenc}
\usepackage{caption}
\usepackage{algorithm}
\usepackage{graphicx,url}
\usepackage{tabularx}
\usepackage{color}
\usepackage{calc}
\usepackage{longtable}
\usepackage{multirow}
\usepackage{array}
\usepackage{hyperref}
\usepackage{tikz}
\usepackage{booktabs}
\usepackage{pgfplotstable}
\usepackage{pgfplots}
\usepackage[outline]{contour}
\usepackage[colorinlistoftodos,prependcaption]{todonotes}
\usepackage{graphicx}
\usepackage[mathscr]{eucal}
\usepackage{hyperref}
\usepackage{bm}

\usepackage{amsmath}

\usepgfplotslibrary{groupplots}
\usetikzlibrary{matrix}
\usetikzlibrary{automata, external, shapes.geometric, fit}
\usetikzlibrary{arrows.meta}
\pgfplotsset{compat=1.13}

\newcommand{\D}{\mathrm{d}}
\newcommand{\E}{\mathrm{e}}
\newcommand*\diff{\mathop{}\!\mathrm{d}}

\newcommand{\vtheta}{\bm{\theta}}

\newcommand{\vsigma}{\bm{\sigma}}

\newcommand{\va}{\bm{a}}
\newcommand{\vb}{\bm{b}}

\newcommand{\ve}{\bm{e}}

\newcommand{\vk}{\bm{k}}

\newcommand{\vq}{\bm{q}}

\newcommand{\vw}{\bm{w}}
\newcommand{\vx}{\bm{x}}
\newcommand{\vy}{\bm{y}}

\newcommand{\mPhi}{\bm{\Phi}}

\newcommand{\mA}{\bm{A}}

\newcommand{\mG}{\bm{G}}

\newcommand{\mI}{\bm{I}}

\newcommand{\mQ}{\bm{Q}}

\newcommand{\sR}{\mathbb{R}}

\definecolor{myblack}{RGB}{53, 53, 53}
\definecolor{myblue}{RGB}{40, 75, 99}
\definecolor{myred}{RGB}{192, 50, 33}
\definecolor{myyellow}{RGB}{255, 166, 48}
\definecolor{mywhite}{RGB}{240, 237, 238}
\definecolor{mygreen}{RGB}{0, 102, 0}

\definecolor{green1}{RGB}{9, 82, 86}
\definecolor{green2}{RGB}{8, 127, 140}
\definecolor{green3}{RGB}{6, 167, 125}
\definecolor{green4}{RGB}{79, 109, 122}
\definecolor{green5}{RGB}{192, 214, 223}
\definecolor{violet}{RGB}{26,69,131}

\definecolor{checkgreen}{rgb}{0,0.6,0}
\definecolor{phase1}{rgb}{0.008,0.655,1.000}
\definecolor{phase2}{rgb}{0.016,0.75,0.700}
\definecolor{phase3}{rgb}{0.929,0.35,0.700}
\definecolor{icsyellow}{cmyk}{0.00,0.11,0.53,0.00}

\usepackage{color, colortbl}
\definecolor{Gray}{gray}{0.9}

\definecolor{checkgreen}{rgb}{0,0.6,0}
\definecolor{phase1}{rgb}{0.008,0.655,1.000}
\definecolor{phase2}{rgb}{0.016,0.75,0.700}
\definecolor{phase3}{rgb}{0.929,0.35,0.700}
\definecolor{icsyellow}{cmyk}{0.00,0.11,0.53,0.00}
\definecolor{green1}{RGB}{9, 82, 86}
\definecolor{green2}{RGB}{8, 127, 140}
\definecolor{green3}{RGB}{6, 167, 125}
\definecolor{green4}{RGB}{79, 109, 122}
\definecolor{green5}{RGB}{192, 214, 223}
\definecolor{violet}{RGB}{26,69,131}
\definecolor{mypurple}{RGB}{150, 0, 180}

\usetikzlibrary{intersections}
\usepgfplotslibrary{fillbetween}
\usepackage{multirow}
\usepackage{stmaryrd}
\usepackage{xfrac}
\usepackage{caption}
\usepackage{subcaption}
\usepackage{lmodern}
\usetikzlibrary{decorations.text}
\usepackage[normalem]{ulem}

\definecolor{blackmy}{RGB}{38, 70, 83}
\definecolor{greenmy}{RGB}{42, 167, 143}
\definecolor{yellowmy}{RGB}{233, 196, 106}
\definecolor{redmy}{RGB}{253,127,111}

\definecolor{bluemy}{RGB}{40, 75, 99}
\definecolor{myred}{RGB}{192, 50, 33}
\definecolor{brownmy}{RGB}{244, 162, 97}

\definecolor{blue1}{RGB}{1, 58, 99}
\definecolor{blue2}{RGB}{70, 143, 175}
\definecolor{blue3}{RGB}{137, 194, 217}

\definecolor{red1}{RGB}{176, 71, 89}
\definecolor{red3}{RGB}{249, 155, 125}
\definecolor{red2}{RGB}{231, 97, 97}

\definecolor{green1}{RGB}{112, 151, 117}
\definecolor{green2}{RGB}{143, 185, 150}
\definecolor{green3}{RGB}{161, 204, 165}

\newcommand{\algorithmiccommentMine}[1]{\bgroup\hfill$\triangleright$~{#1}\egroup}

\newcommand\oldtext[1]{}
\newcommand\cancel[1]{}

\begin{document}

\maketitle

\begin{abstract}
    Neural network-based methods have emerged as powerful tools for solving partial differential equations (PDEs) in scientific and engineering applications, particularly when handling complex domains or incorporating empirical data. These methods leverage neural networks as basis functions to approximate PDE solutions. However, training such networks can be challenging, often resulting in limited accuracy. In this paper, we investigate the training dynamics of neural network-based PDE solvers with a focus on the impact of initialization techniques. We assess training difficulty by analyzing the eigenvalue distribution of the kernel and apply the concept of effective rank to quantify this difficulty, where a larger effective rank correlates with faster convergence of the training error.
    Building upon this, we discover through theoretical analysis and numerical experiments that two initialization techniques, partition of unity (PoU) and variance scaling (VS), enhance the effective rank, thereby accelerating the convergence of training error. Furthermore, comprehensive experiments using popular PDE-solving frameworks, such as PINN, Deep Ritz, and the operator learning framework DeepOnet, confirm that these initialization techniques consistently speed up convergence, in line with our theoretical findings.
\end{abstract}



\section{Introduction}
    Deep learning techniques have seen extensive applications in various domains such as computer vision~\cite{cv1,cv2}, natural language processing~\cite{nlp1,nlp2}, and AI for science~\cite{ai4sci1,ai4sci2}. In recent years, there has been significant progress in leveraging neural network (NN) architectures in the fields of mathematics and engineering. One notable area of advancement is the use of neural network structures for solving partial differential equations (PDEs), particularly when dealing with complex domains and integrating empirical data. Various methodologies, such as Physics-Informed Neural Networks (PINNs)~\cite{PINNori}, Deep Ritz~\cite{DeepRitz}, Weak Adversarial Networks~\cite{wan}, and random feature methods for solving PDEs~\cite{chen2022bridging,rnn1}, have shown promising potential in solving PDEs. These advancements in utilizing neural networks for solving PDEs have opened up new avenues in scientific computing.

    The utilization of NN-based methods for solving partial differential equations (PDEs) is primarily based on approximating the solutions to PDEs using NN functions, represented as:

\begin{equation}
    u_{\vtheta}(\vx):=\phi(\vx;\vtheta) = \sum_{j=1}^M a_j \phi_j(\vx) ,~\vx\in\Omega.
\end{equation}
Here $\phi_j(\vx)$ denotes the NN functions which can be regarded as basis functions and $a_j$ are the corresponding coefficients. The symbol $\Omega$ denotes the domain over which the PDE is being solved. The universal approximation theorem and its extensions~\cite{universal3,universal4,universal1,DeepOnetN,universal5,universal7} provide theoretical possibility for NN functions to approximate the solutions to PDEs. However, in practice, training NNs to accurately approximate PDE solutions can be challenging, as gradient-based algorithms often struggle to find the corresponding solutions ~\cite{failurepinn1,failurepinn2,loss2}. Prior work~\cite{failurepinn2} has quantified and analyzed the stiffness of the gradient flow dynamics, elucidating the challenges of training PINNs via gradient descent. The authors primarily focus on quantifying the maximum eigenvalue of the Hessian of loss but do not extensively analyze the relationship between the eigenvalues and the convergence of the loss function. In their subsequent work~\cite{loss2}, they derive the neural tangent kernel (NTK) of PINNs and demonstrate that the convergence rate of the total training error of a PINN model can be analyzed based on the spectrum of its NTK at initialization. They propose using the mean of all eigenvalues of the NTK to measure the average convergence rate of the training error. However, this metric is not particularly suitable for assessing the performance of NN-based PDE solvers. Similarly, the analysis of eigenvalues in the context of comparing automatic differentiation and finite difference methods is mentioned in~\cite{chen2024automatic}, but the focus is primarily on the impact of small eigenvalues on the training error, rather than on the overall convergence behavior.

Ongoing research efforts and the development of specialized techniques aim to address these challenges and fully harness the potential of NN-based methods for accurately and efficiently solving PDEs. This involves three key aspects: initialization techniques, NN structures, and training methods. Initialization techniques involve the selection of appropriate methods for initializing network parameters~\cite{init1,init2} and defining computational domains~\cite{pou1,pou2,pou3}. The NN structures require careful consideration of selecting suitable architectures~\cite{structure1,structure2,structure3,lan2023dosnet,structure7} and effective loss functions~\cite{failurepinn2,loss2}. Training methods include identifying proper strategies~\cite{train1,train2,strategies4} and optimization methods~\cite{optimiz1,optimiz2,optimiz3} to enhance the training process and achieve optimal performance.

The recently proposed method of utilizing random feature methods for solving PDEs has demonstrated remarkable accuracy. This approach mainly involves using random feature functions $\phi_{j}(\vx)$ as basis functions and transforming the problem of solving PDEs into a problem of solving a linear system of equations to obtain the coefficients $a_j$. A NN structure with two layers referred to as random feature models~\cite{chen2022bridging,rfm1,rfm2,rfm4}, whereas a NN with multiple layers called a random neural network~\cite{rnn1,rnn2}. By employing direct matrix-based methods instead of gradient descent-based methods, this approach addresses training challenges and achieves higher accuracy in PDE solvers. However, it is important to note that this method is primarily applicable to linear PDEs and may have limitations when solving nonlinear PDEs. The method utilizes two initialization techniques, namely variance scaling (VS) of the inner parameters in $\phi_{j}(\vx)$ and partition of unity (PoU)~\cite{chen2022bridging,rfm4}. However, the article does not provide a comprehensive analysis of how these two strategies specifically impact the final accuracy of the solution.

Motivated by the aforementioned developments, our research focuses on the training perspective, specifically investigating the influence of the kernel's eigenvalue distribution on convergence and the impact of initialization techniques on training difficulty. During training, the evolution of the loss function, typically the mean square error $L(\vtheta)= \ve^{\top} \ve$, is described by:

\begin{equation}
\frac{\D  L(\vtheta)}{\D  t}=-\ve^{\top}  \mG \ve,
\end{equation}
where $\mG$ is the kernel (see more details in Eq.~\eqref{eq..GradFlowLoss} and Eq.~\eqref{eq..Gram}).
Our research is primarily dedicated to exploring the dynamics of training NNs for solving PDEs from a rigorous perspective. Specifically, we investigate the impact of the eigenvalue distribution of the kernel $\mG$ on the convergence behavior of the training error. Additionally, we analyze the efficacy of two key techniques employed during the initialization phase, namely partition of unity (PoU) and variance scaling (VS), in enhancing the accuracy of the solution from a training standpoint. By delving into these aspects, we provide valuable insights into the training dynamics of NN-based PDE solvers and offer a deeper understanding of their convergence behavior and accuracy improvement strategies. Our main contributions are summarized as follows:
\begin{itemize}
    \item We quantify the training difficulty by using the concept of \textit{effective rank}, which captures the eigenvalue distribution of the kernel $\mG$ associated with the gradient descent method. We find that higher effective rank corresponds to easier training, as indicated by faster convergence during the optimization process using gradient descent-based methods.
    \item  We analyze two initialization techniques: PoU and VS. Through numerical experiments and theoretical analyses on the random feature model, we demonstrate the impact of these two initialization techniques on effective rank.
    \item To validate our analysis, we conduct a comprehensive suite of experiments using three distinct models: PINNs, Deep Ritz, and DeepONets. Across all models, the experiments consistently demonstrate that the two initialization techniques, PoU and VS, accelerate the convergence of training error.
\end{itemize}

The paper is structured as follows. In Section 2, we provide a review of the framework for utilizing NN structures to solve PDEs, and the random feature methods focusing on two specific initialization techniques. In Section 3, we derive the training dynamics and introduce the key concept: effective rank. Subsequently, we concentrate on the random feature model (RFM) as a case study to analyze the influence of PoU and VS techniques on effective rank, and how this impacts the training process. In Section 4, we present experimental results to validate our analysis through three distinct models: PINN, Deep Ritz and DeepOnet.

\section{Preliminaries}

\subsection{Solving PDEs with NNs.}
In the context of solving partial differential equations (PDEs), we consider the problem formulated as follows:

\begin{equation}
\begin{cases}\mathcal{L} u(\vx)=f(\vx), & \vx \in \Omega, \\ \mathcal{B} {u}(\vx)={g}(\vx), & \vx\in \partial \Omega,\end{cases}
\end{equation}
where $\mathcal{L}$ is a differential operator representing the PDE, $\vx=\left(x_1, \cdots, x_d\right)^{\top}$, $\Omega$ is bounded and connected domain in $\sR^d$ and $\mathcal{B} {u}(\vx)={g}(\vx), \vx\in \partial \Omega$ stands for the boundary condition. In the context of solving PDE problems using NN structures, a commonly used architecture is as follows:

\begin{equation}
\phi(\vx;\vtheta) =\sum_{i=1}^M a_j \phi_j(\vx) = \sum_{j=1}^M a_j \sigma(\vk_j \cdot \vx + b_j) = \sum_{i=1}^M a_j \sigma(\vw_j \cdot \vx).
\label{Eq.NNstructure}
\end{equation}
 In this paper, we mainly focus on the initialization of $\vw_j \sim U(-R_m, R_m)$ indicating that $\vw_j$ is randomly selected from a uniform distribution with a range $[-R_m, R_m]$. With a little bit abuse of notation, here we augment $\vx$ to $(\vx,1)^{\top}$ (and still denote it as $\vx$) and let $\vw_j=(\vk_j,b_j)^{\top}$. In this NN structure, $\{a_j\}_{j=1}^{M}$ represents the output layer parameters of the network, and $\phi_{j}(\vx)$ denotes the basis function, which is composed of a combination of nonlinear activation functions and linear layers. Typically, we consider a simple two-layer NN as shown above $\phi_{j}(\vx) = \sigma(\vk_j \cdot \vx + b_j)$, where $\sigma$ represents the nonlinear activation function. In the NN paradigm, the parameters $\{(a_j, \vk_j, b_j)\}_{j=1}^M\subset \sR\times\sR^d\times\sR$ can be trained by minimizing the loss function. The loss function can be set in different cases, take the physics-informed neural network (PINN)~\cite{PINNori} as an example, the loss function can be defined as:

\begin{equation}
    L(\vtheta) := L_{\text{res}}(\vtheta) + \gamma L_{\text{bc}}(\vtheta).
\end{equation}

Here, $L(\vtheta)$ represents the total loss, which consists of two components: the PDEs residual term  $L_{\text{res}}(\vtheta) = \|\mathcal{L} \phi(\vx;\vtheta)-f(\vx)\|^2_2$ and the boundary condition term $L_{\text{bc}}(\vtheta) = \|\mathcal{B} \phi(\vx;\vtheta) - g(\vx)\|_2^2$. The parameter $\gamma$ serves as a regularization factor to balance the influence of these two loss terms. Specifically, the loss function has the form like:

\begin{equation}
    L(\vtheta) := L_{\text{res}}(\vtheta) + \gamma L_{\text{bc}}(\vtheta) =\sum_{i=1}^N |\mathcal{L}\phi(\vx_i;\vtheta) - f(\vx_i)|^2 + \gamma \sum_{i=1}^{N^{\prime}} |\mathcal{B}\phi(\vx_i^{\prime};\vtheta)-g(\vx_i^{\prime})|^2,
\label{PINN}
\end{equation}
where the variables
$\vx_j \in \Omega$ and $\vx^{\prime}_j \in \partial\Omega$. The points $\vx_j$ and $\vx^{\prime}_j$ can be randomly selected from within the domain $\Omega$ and its boundary $\partial \Omega$, or they can be chosen from uniform grids.
In the case of Deep Ritz method~\cite{DeepRitz}, the loss function is the corresponding variation form of the PDEs.

\subsection{Solving PDEs with random feature methods}

The random feature model (RFM) maintains the structural similarity to Eq.~\eqref{Eq.NNstructure}, with the key distinction that the inner parameters $\{\vk_j, b_j\}_{j=1}^M$ are randomly initialized and subsequently fixed, as described in ~\cite{chen2022bridging}. A common choice for the initialization distribution is the uniform distribution, where $\vk_{j} \sim U\left([-R_{m}, R_{m}]^{d}\right)$ and $b_{j} \sim U\left([-R_{m}, R_{m}]\right)$, although alternative distributions can be employed. Notably, only the outer parameters $\{a_j\}_{j=1}^M$ are subject to training and remain free to adapt during the optimization process. Therefore, the solution to problem a can be obtained through the use of direct matrix solving methods, such as the QR method~\cite{QRmethod} and least square method~\cite{leastsquare}. Another important technique utilized in the method is the PoU. To demonstrate the application of the RFM, we consider a one-dimensional case where the NN function approximating the solutions to PDEs takes the following form:

\begin{equation}
u_M(\vx)=\sum_{n=1}^{M_p} \psi_n(\vx) \sum_{j=1}^{J_n} a_{n j} \phi_{n j}(\vx) = \sum_{n=1}^{M_p} \psi_n(\vx) \sum_{j=1}^{J_n} a_{n j} \sigma(\vk_{nj} \cdot \vx + b_{nj}),
\end{equation}
where $\psi_n(\vx)$ represents the construction function and $\phi_{n j}(\vx)=\sigma(\vk_{nj} \cdot \vx + b_{nj})$ is the random feature functions and $\vk_{nj} \sim U\left([-R_{m}, R_{m}]^{d}\right)$ and $b_{nj} \sim U\left([-R_{m}, R_{m}]\right)$, $a_{nj}$ is the coefficients. $M_p$ is number of partition of unity and $J_n$ is number of bases in each unity. In one dimensional case, the construction function can be:
\begin{equation}
\psi_n^a(x)=\chi_{\{-1 \leq \tilde{x}<1\}},
\end{equation}
or
\begin{equation}
\psi_n^b(x)= \begin{cases}\frac{1+\sin (2 \pi \tilde{x})}{2}, & -\frac{5}{4} \leq \tilde{x}<-\frac{3}{4}, \\
1, & -\frac{3}{4} \leq \tilde{x}<\frac{3}{4}, \\
\frac{1-\sin (2 \pi \tilde{x})}{2}, & \frac{3}{4} \leq \tilde{x}<\frac{5}{4}, \\
0, & \text { otherwise},\end{cases}
\end{equation}
where $\chi_{\{-1 \leq \tilde{x}<1\}}$ is the characteristic function, $\tilde{x} = \frac{1}{r_n}(x - x_n)$ represents a linear transformation to the interval $[-1,1]$, and $r_n$ is a normalizing factor and $x_n$ is the center associated with the $n$-th local solution as shown in Fig.~\ref{Fig.PoU Construction function}(a). In $d$ dimension case, we directly employ $\psi_n(\vx) = \prod_{k=1}^{d}\psi_n(x_k)$ as shown in Fig.~\ref{Fig.PoU Construction function}(b). The PoU is a domain decomposition method that is utilized in the context of NN-based PDE solvers ~\cite{pou1,pou2,pou3}. However, the specific reasons why applying this strategy can lead to improved accuracy have not yet been thoroughly investigated.

\begin{figure}[!ht]
    \centering
    \setcounter {subfigure} 0(a){
    \includegraphics[scale=0.4]{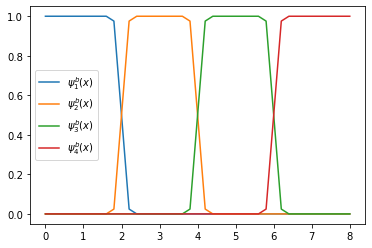}}
    \centering
    \setcounter {subfigure} 0(b){
    \includegraphics[scale=0.32]{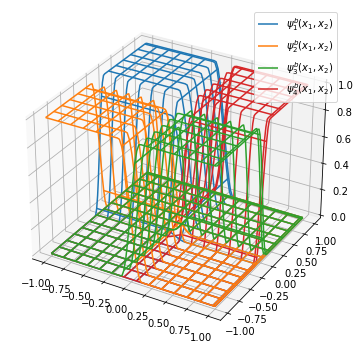}}
    \caption{Construction functions $\{\psi^{b}_n(x),n=1,\dots,M_p\}$. (a) 1D Construction functions $\psi_n^{b}(x)$ define on $[0,8]$ for dividing the region into four domains ($M_p = 4$). In this case $x_n = \{1,3,5,7\}$ with corresponding $r_n = \frac{8-0}{4} = 2$. (b) 2D Construction functions $\psi_n^{b}(x_1,x_2)$ define on $[-1,1]^2$ for dividing the region into four domains ($M_p = 4$).}
    \label{Fig.PoU Construction function}
\end{figure}

Based on this, in this paper, we analyze these two initialization techniques, PoU and VS, from a training perspective to study how they improve the convergence of the training process.

\section{Eigenvalue-based training difficulty quantification and acceleration}

In this section, we begin by analyzing the impact of the distribution of kernel eigenvalues on the convergence rate of training dynamics. We employ the effective rank~\cite{roy2007effective} as a measure of the convergence rate. Through our investigations, we discover that two initialization techniques, namely PoU and VS, can effectively increase the effective rank and, consequently, improve the convergence rate of the training process. We primarily investigate the relationship between two quantities and $\operatorname{erank}(\mG)$ one is $M_p$, the number of PoU, and the other is $R_m$, the VS of initialization.

\subsection{Effective rank in training dynamics}

Without loss of generality, we only consider the PDEs residual term $L_{\text{res}}(\vtheta)$ and linear PDEs case here, i.e.,
\begin{equation}
   L_{\text{res}}(\vtheta) = \| \mathcal{L} \phi(\vx;\vtheta)-f(\vx)\|^2_2 = \bigg\|\mathcal{L} \left(\sum_{k=1}^M a_k \sigma(\vw_k \cdot \vx) \right)-f(\vx) \bigg\|^2_2. 
\label{Eq.PINNLoss}
\end{equation}
The training dynamics based on gradient descent (GD) at the continuous limit obeys the following gradient flow of $\vtheta$,

\begin{equation}
\frac{\D \vtheta}{\D t}=-\nabla_{\vtheta} L_{\text{res}}(\vtheta) .
\end{equation}

More precisely, $\vtheta=\operatorname{vec}\left(\left\{\vq_k\right\}_{k=1}^M\right)$ with $\vq_k=\left(a_k, \vw_k^{\top}\right)^{\top}, k \in[M]$ solves
$$
\begin{aligned}
\frac{\D a_k}{\D t} & =-\sum_{i=1}^N \mathcal{L} \bigg( \frac{1}{\sqrt{M}}\sum_{k=1}^M\sigma\left(\vw_k \cdot \vx_i\right)\bigg)\left( \mathcal{L}\bigg(\frac{1}{\sqrt{M}} \sum_{k=1}^M a_k  \sigma\left(\vw_k \cdot  \vx_i\right)\bigg)-f(\vx_i)\right), \\
\frac{\D \vw_k}{\D t} & =- \sum_{i=1}^N \mathcal{L} \bigg(  \frac{1}{\sqrt{M}} \sum_{k=1}^M a_k \sigma^{\prime}\left(\vw_k \cdot  \vx_i\right) \vx_i\bigg)\left( \mathcal{L}\bigg(\frac{1}{\sqrt{M}} \sum_{k=1}^M a_k  \sigma\left(\vw_k \cdot  \vx_i\right)\bigg)-f(\vx_i)\right).
\end{aligned}
$$
Thus for the gradient flow of the loss function, we have

\begin{equation}\label{eq..GradFlowLoss}
\frac{\D L_{\text{res}}(\vtheta)}{\D t} = \nabla_{\va}L_{\text{res}}(\vtheta)\frac{\D \va}{\D t} + \nabla_{\vw}L_{\text{res}}(\vtheta)\frac{\D \vw}{\D t}=-2\ve^{\top} \mG \ve  =-2\ve^{\top} (\mG^{[\va]} + \mG^{[\vw]})  \ve ,
\end{equation}
where
\begin{equation}\label{eq..Gram}
\begin{aligned}
&\va=\operatorname{vec}\left(\left\{a_j\right\}_{j=1}^M\right), \vw=\operatorname{vec}\left(\left\{\vw_j\right\}_{j=1}^M\right), \\ 
& [\mG^{[\va]}(\vtheta)]_{ij}=\frac{1}{ M} \sum_{k=1}^M \big(\mathcal{L}\sigma\left(\vw_k \cdot \vx_i\right) \big)\big(\mathcal{L}\sigma\left(\vw_k \cdot \vx_j\right) \big), \\
& [\mG^{[\vw]}(\vtheta )]_{i j}=\frac{1}{M} \sum_{k=1}^M a_k^2 \big(\mathcal{L}\sigma^{\prime}\left(\vw_k \cdot \vx_i\right)\big) \big(\mathcal{L}\sigma^{\prime}\left(\vw_k^{\top} \vx_j\right)\big) \vx_i \cdot \vx_j, \\
& [\ve]_i = \left( \mathcal{L}\bigg(\frac{1}{\sqrt{M}} \sum_{k=1}^M a_k \cdot \sigma\left(\vw \cdot  \vx_i\right)\bigg)-f(\vx_i)\right).
\end{aligned}
\end{equation}

Note that
$L_{\text{res}}(\vtheta) = \ve^{\top}\ve.$ We then explore the eigenvalue decomposition of the matrix $\mG = \mQ^{\top}\text{diag}\{\lambda_1,\lambda_2,\dots,\lambda_N\}\mQ$, where $\lambda_1 \geq \lambda_2 \geq \dots \geq \lambda_N$. In this context, $\widetilde{\ve}_i = [\mQ \ve]_i$ represents the $i$-th column of the matrix product $\mQ \ve$. With this, we can rewrite the gradient flow as:

\begin{equation}
\begin{aligned}
\frac{\D L_{\text{res}}(\vtheta)}{\D t} &= -2\ve^{\top} \mG \ve
= -2 \sum_{i=1}^N \lambda_i \widetilde{\ve}_i^2,
\end{aligned}
\end{equation}
where $\ve_i^2 = \widetilde{\ve}_i^2$, and due to $\mQ^{\top} \mQ = \mI$, thus $\sum_{i=1}^N \widetilde{\ve}_i^2 = L_{\text{res}}(\vtheta)$. 
Intuitively, eigenvectors corresponding to larger eigenvalues converge at a faster rate, while eigenvectors associated with smaller eigenvalues converge more slowly. This phenomenon is closely related to the concept of frequency principle (also known as spectral bias), a well-known pathology that hinders deep fully-connected networks from learning high-frequency functions~\cite{spetralbias1,spetralbias2,spetralbias3,spetralbias4}. In the context of NN-based PDE solvers, the eigenvectors corresponding to higher eigenvalues of the kernel $\mG$ generally exhibit lower frequencies, aligning with the spectral bias phenomenon. Consequently, the overall convergence rate of the loss function is limited by the smallest eigenvalue. However, due to the inherent limitations of gradient descent methods, within a finite number of steps, the descent of eigenvectors corresponding to smaller eigenvalues is slower, and their contribution to the decrease in loss is limited. Furthermore, to ensure training stability in gradient descent (GD) methods, the step size $\Delta t$ in gradient descent is often constrained by the largest eigenvalue. Based on this analysis, we conclude that the convergence rate of the training error is closely linked to the distribution of eigenvalues. In this context, the effective rank~\cite{roy2007effective} provides a more comprehensive measure, capturing the influence of both large and small eigenvalues on the convergence behavior.

\begin{definition}
[\textbf{Effective rank}]
For a matrix $\mA\in\sR^{m\times n}$ $(m>n)$ with singular values $\sigma_1\ge\sigma_2\ge\ldots\ge\sigma_{n}$, the truncation entropy of $\mA$ is defined as:
\begin{equation}
    \operatorname{erank}(\mA) = \exp \left\{-\sum_{k=1}^{n} p_k \log p_k,\right\},
\end{equation}
where $p_k=\frac{\sigma_k}{\|\vsigma\|_1}$, for $k=1,2,\ldots, n$, with $\vsigma = [\sigma_1,\sigma_2,\dots,\sigma_n]^{\top}$.
\end{definition}

Here, the concept of effective rank, as proposed in~\cite{roy2007effective}, aims to provide better continuity compared to the definition of rank, enabling minimization in certain signal processing applications. In this paper, we utilize the effective rank to assess the distribution of eigenvalues of the kernel $\mG$. Within this definition, the quantity $-\sum_{k=1}^{n} p_k \log p_k$ represents entropy. A larger effective rank corresponds to the values of $p_k$ being closer to each other. As a result, the matrix $\mA$ becomes closer to a scalar matrix.

To provide a more intuitive illustration of the relationship between effective rank and convergence rate, we consider a simple problem for solving a linear system $\mA\vx = \vb$, where $\mA$ is a diagonal matrix given by $\text{diag}\{\lambda_1,\lambda_2,\dots,\lambda_N\}$. In this case, we set $\lambda_1 = 256$ and $\lambda_N = 1$, which means that different $\mA$ share the same largest eigenvalue and condition number but different effective rank. The vector $\vb$ is randomly chosen and then remains the same for different $\mA$. We utilize gradient descent to minimize the mean squared error (MSE) loss between $\mA\vx$ and $\vb$ over 100 epochs, with a learning rate of $\texttt{5e-2}$. In Fig.~\ref{Fig.Kernel_eigen}(a), we present the results for different eigenvalue distributions, highlighting their corresponding effective ranks. The graph demonstrates that a larger effective rank is achieved when the eigenvalue distribution is relatively uniform.  Fig.~\ref{Fig.Kernel_eigen}(b) illustrates the convergence of each dimension $x_i$, where $i=1,2,\dots,N$. The results depicted in the loss curve of Fig.~\ref{Fig.Kernel_loss} indicate that a larger effective rank results in faster convergence of the loss.

\begin{figure}[!ht]
    \centering
    \setcounter {subfigure} 0(a.1){
    \includegraphics[scale=0.48]{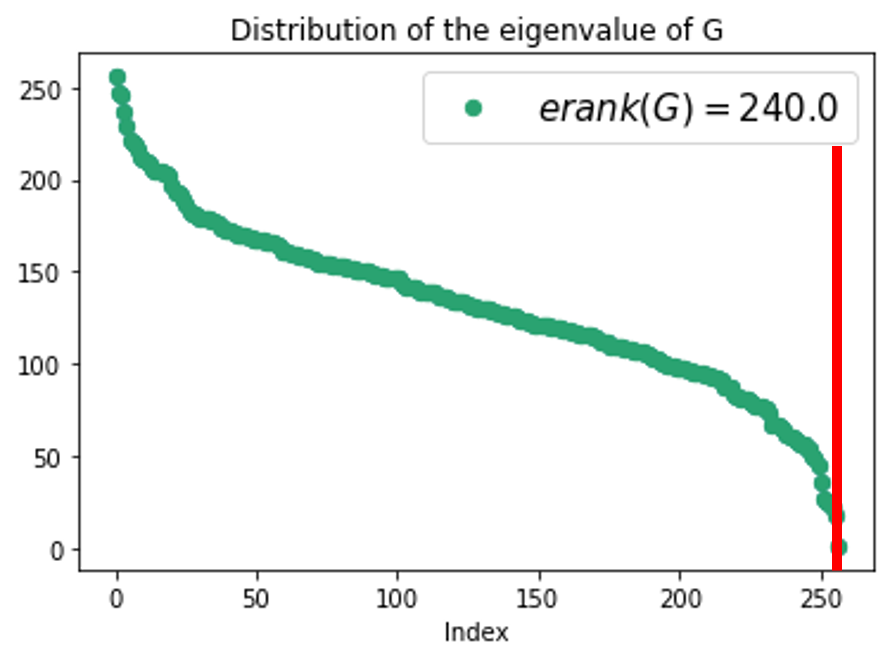}}
    \setcounter {subfigure} 0(a.2){
    \includegraphics[scale=0.48]{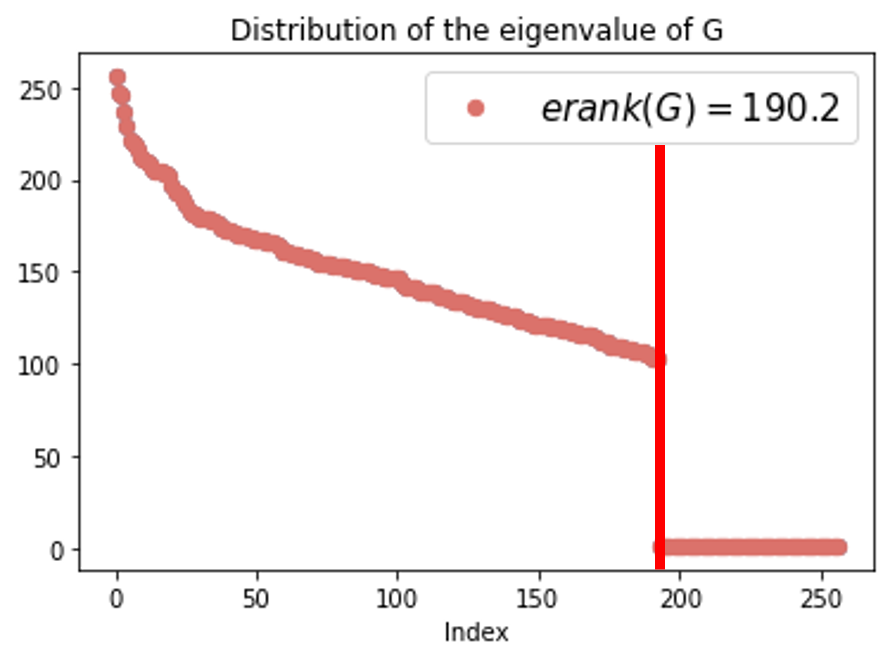}}
    \setcounter {subfigure} 0(a.3){
    \includegraphics[scale=0.48]{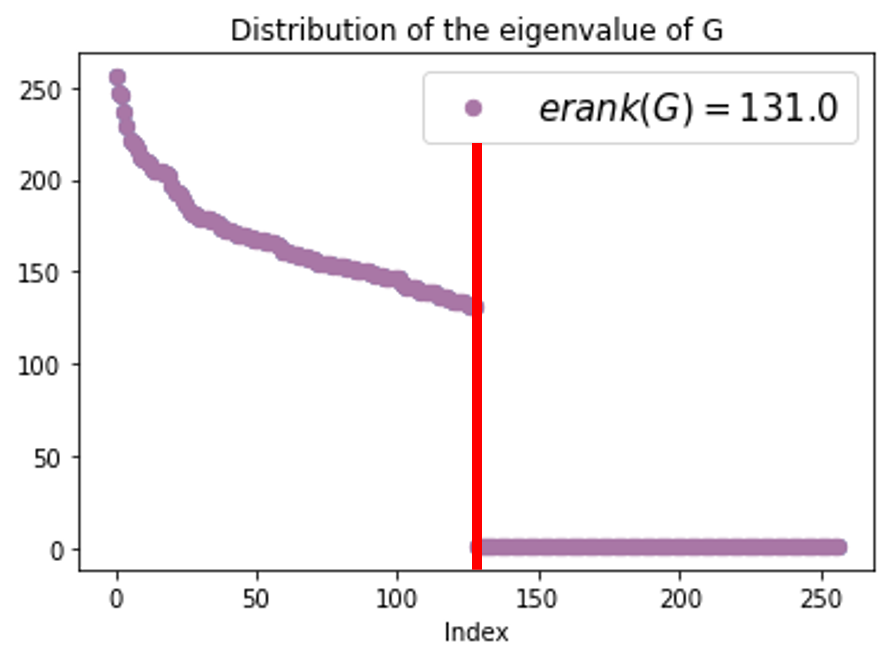}}
    \setcounter {subfigure} 0(b.1){
    \includegraphics[scale=0.48]{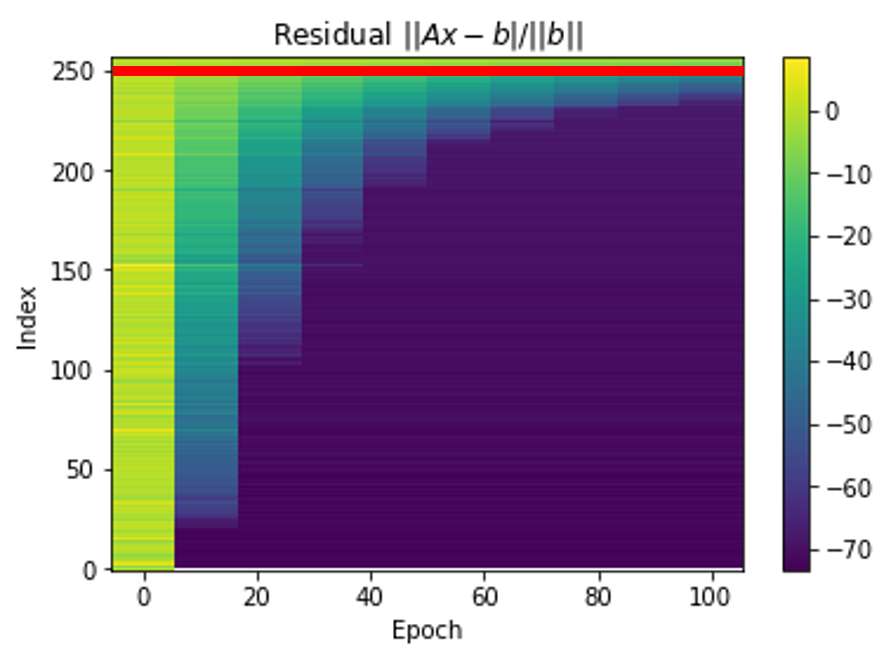}}
    \setcounter {subfigure} 0(b.2){
    \includegraphics[scale=0.48]{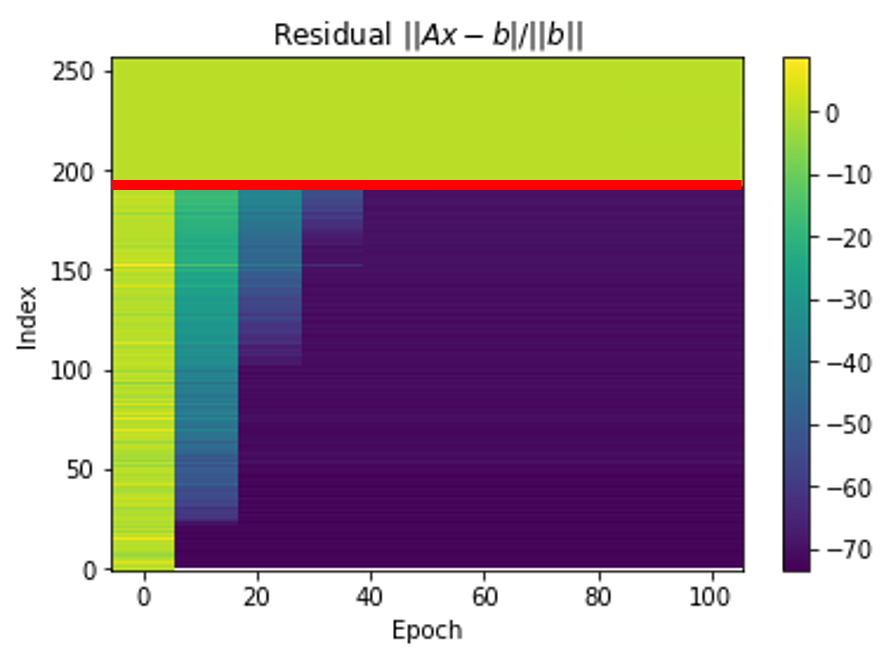}}
    \setcounter {subfigure} 0(b.2){
    \includegraphics[scale=0.48]{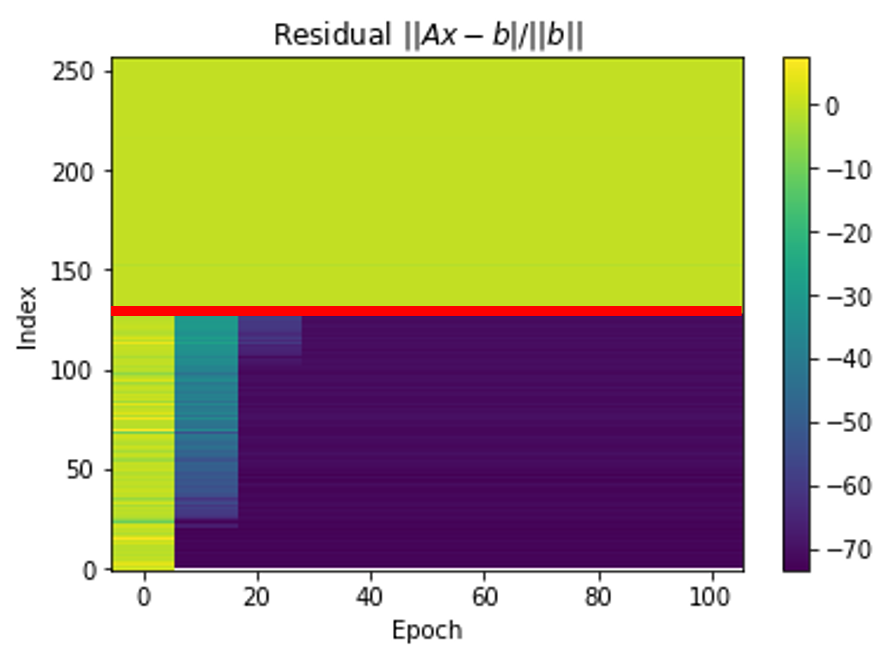}}
    \caption{ Results for the solving $\mA\vx = \vb$, where $\mA = \text{diag}\{\lambda_1,\lambda_2,\dots,\lambda_N\}$.  \textbf{(a)}: Eigenvalue distribution of matrix A along with the corresponding effective rank. \textbf{(b)}: Convergence curves for each dimension $x_i$, $i=1,2\dots N$. The red line indicates that the $x_i$ corresponding to the eigenvalues before the line have converged after 100 epochs.}
    \label{Fig.Kernel_eigen}
\end{figure}

\begin{figure}[!ht]
		\centering
		\includegraphics[width= 0.45\textwidth]{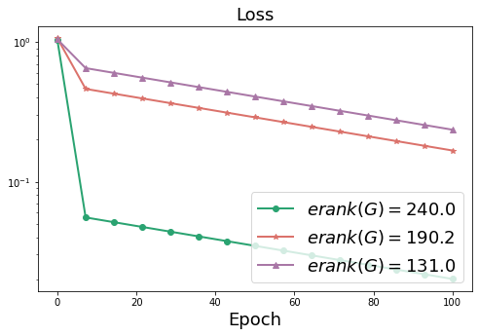}
\caption{Loss curve for solving $\mA\vx =\vb$ with mean square loss and the effective rank corresponding to different cases.}
\label{Fig.Kernel_loss}
\end{figure}

The above discussion and empirical study show that the effective rank reflect the convergence rate of the loss function. Specifically, a larger effective rank indicates faster convergence. Therefore, it is crucial to explore methods to improve the effective rank. 

\subsection{Initialization techniques to enhance effective rank}

In the following discussion, we focus on the impact of initialization techniques on the effective rank. We demonstrate that two specific techniques, PoU and VS, as described in~\cite{chen2022bridging}, can significantly improve the effective rank. This finding is supported by both experimental evidence and theoretical analysis. Notably, we employ the random feature method~\cite{chen2022bridging} as an exemplar case due to its shallower network structure, where the non-linear layers remain fixed, resulting in a corresponding fixed kernel denoted as $\mG = \mG^{[\va]}$. While our results can be extended to deep NN structures~\cite{rnn1,rnn2}, we focus on a two-layer fully connected NN structure for illustrative purposes.

\subsubsection{Partition of unity (PoU)}

Partition of unity (PoU) is a domain decomposition technique that can be implemented by multiplying a construction function with the NN function. The modified NN function can be expressed as follows:

\begin{equation}
\phi(\vx;\vtheta)  =\sum_{n=1}^{M_p} \psi_n(\vx) \sum_{j=1}^{J_n} \phi_{n j}(\vx) = \sum_{n=1}^{M_p} \psi_n(\vx) \sum_{j=1}^{J_n}a_{n j} \sigma(\vk_{nj} \cdot \vx + b_{nj}), 
\end{equation}
where $\psi_n(\vx)$ is called the construction function and $M_p$ stands for the number of partition of unity.  Fig.~\ref{Fig.basis_Mp} depicts the modified NN basis functions $\{\psi_n(x)\sigma(\vw_{nj}\cdot \vx)\}$ for different values of $M_p$. As depicted in the figure, as $M_p$ increases, the modified NN basis functions become more localized meaning they have local support.

\begin{figure}[!ht]
    \centering
    \setcounter {subfigure} 0(a){
    \includegraphics[scale=0.16]{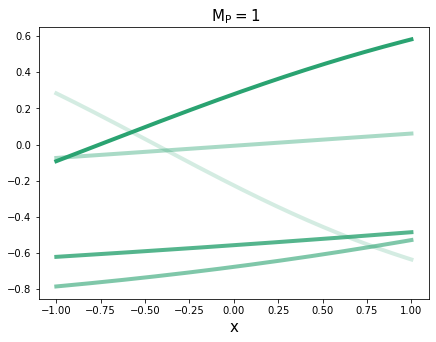}}
    \setcounter {subfigure} 0(b){
    \includegraphics[scale=0.16]{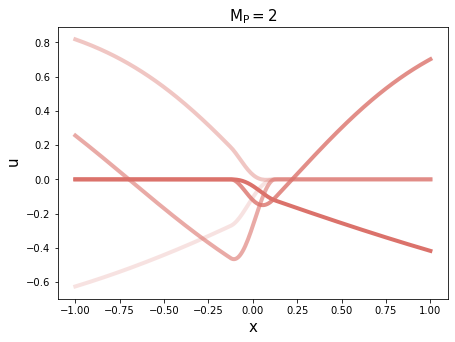}}
    \setcounter {subfigure} 0(c){
    \includegraphics[scale=0.16]{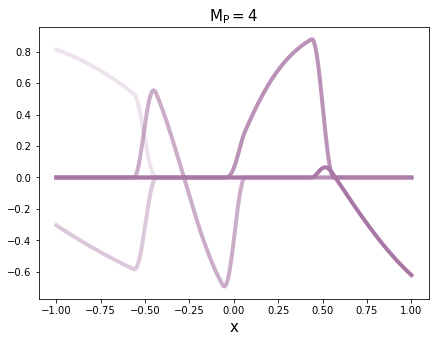}}
    \setcounter {subfigure} 0(d){
    \includegraphics[scale=0.16]{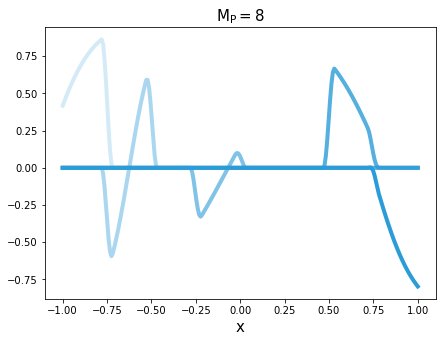}}
    \caption{Basis functions defined on $[-1,1]$, $\{\psi_n(x)\sigma(\vw_{nj}\cdot \vx), n=1,2\dots M_p,j=1,2,\dots,J_n\}$ for different numbers of PoU $M_p = \{1,2,4,8\}$. In each panel (with a fixed $M_p$), five basis functions are plotted with different shades of color.}
    \label{Fig.basis_Mp}
\end{figure}

To analyze this, we start with a two-layer random feature functions with $M$ hidden neurons problem, i.e.,

\begin{equation}\label{eq::NeuProb}
        \phi(x;\vtheta) = \frac{1}{\sqrt{M}}\sum_{k=1}^{M}a_k\sigma(w_k x+b_k),
\end{equation}
defined on a one-dimension region $I := [-1,1]$. Activation function $\sigma$ can be chosen as $\tanh(x)$ or $\sin(x)$, with input $\{x_i\}_{i=1}^N$  equally distributed selected from $I$ and $w_k, b_k\sim U[-1,1]$. For simplicity of notation, define $\vw_k=(w_k,b_k)^{\top}$, $\vx=(x,1)^{\top}$ and $\vx_i=(x_i,1)^{\top}$. Hence Eq.~\eqref{eq::NeuProb} is equivalent to
 \begin{equation}\label{eq::NeuProbSimplify}
     \phi(\vx;\vtheta) = \frac{1}{\sqrt{M}}\sum_{k=1}^{M}a_k\sigma(\vw_k\cdot\vx).
 \end{equation}
 
For simplicity, here we take the construction function $\psi_a(x)$ and $M_p = 2$ as an example. In this way the domain $I$ can be divided into two parts $I^\mathrm{L}:=[-1,0)$ and $I^\mathrm{R}:=[0,1]$ hence $I=I^\mathrm{L}\cup I^\mathrm{R}$ and center point of $I^\mathrm{L}$ and $I^\mathrm{R}$ is $x^\mathrm{L}$ and $x^\mathrm{R}$ respectively. Moreover, define $\vx^\mathrm{L}=(x^\mathrm{L},\frac{1}{2})$ and $\vx^\mathrm{R}=(x^\mathrm{R},\frac{1}{2})$. 

\begin{equation}
        \begin{aligned}
            \phi_{\mathrm{PoU}}(\vy,\vtheta) &= \frac{1}{\sqrt{M}}\sum_{n=1}^{2} \psi_n^{a}(\vy) \sum_{k=1}^{M/2}  a_{nk}\sigma(\vw_k \cdot \tilde{\vy})
            \\&= \frac{1}{\sqrt{M}}\left(\sum_{k=1}^{M/2}a^{\mathrm{L}}_k\sigma\left(\vw^{\mathrm{L}}_k\cdot 2(\vy-\vx^\mathrm{L})\right)\chi_{I^\mathrm{L}}+\sum_{k=1}^{M/2}a^\mathrm{R}_k\sigma\left(\vw^\mathrm{R}_k\cdot 2(\vy-\vx^\mathrm{R})\right)\chi_{I^\mathrm{R}}\right)
            \\ &= \frac{1}{\sqrt{M}}\left(\sum_{k=1}^{M/2}a^{\mathrm{L}}_k\sigma\left(\vw^{\mathrm{L}}_k \cdot\vy^{\mathrm{L}}\right)\chi_{I^\mathrm{L}}+\sum_{k=1}^{M/2}a^\mathrm{R}_k\sigma\left(\vw^\mathrm{R}_k\cdot \vy^\mathrm{R}\right)\chi_{I^\mathrm{R}}\right) ,
        \end{aligned}
\end{equation}
    where $\chi_{I^\mathrm{L}}$, $\chi_{I^\mathrm{R}}$ are characteristic functions, $\vw^{\mathrm{L}},\vw^{\mathrm{R}}\sim U([-1,1]^2),$ $\vy^{\mathrm{L}}:=2(\vx-\vx^\mathrm{L})$, $\vy^\mathrm{R}:=2(\vx-\vx^\mathrm{R})$, $\vy^{\mathrm{L}}, \vy^{\mathrm{R}}\in ([-1,1],1)^{\top}$, and $\{a_{\mathrm{PoU}}\}_{k=1}^{M}:=\{a_k^{\mathrm{L}}\}_{k=1}^{M/2}\cup \{a_k^\mathrm{R}\}_{k=1}^{M/2}$.

We consider the basis matrix defined as
    \begin{equation}
        \Phi_{ik}=\frac{1}{\sqrt{M}}\sigma(\vw_k\cdot \vx_i) \quad \text{and}\quad  [\mPhi\mPhi^{\top}]_{ij}=\frac{1}{M}\sum_{k=1}^{M}\sigma(\vw_k\cdot \vx_i)\sigma(\vw_k\cdot \vx_j),
    \end{equation} 
for the PoU strategies, we can see the following matrix can be decomposed by two parts $\mPhi_{\mathrm{PoU}}=\mPhi_{\mathrm{PoU}}^\mathrm{L}\oplus \mPhi_{\mathrm{PoU}}^\mathrm{R}$, $\mPhi_{\mathrm{PoU}}^\mathrm{L}, \mPhi_{\mathrm{PoU}}^\mathrm{R}\in \sR^{N/2\times M/2}$, where

\begin{equation}
\begin{aligned}   
[\mPhi_{\mathrm{PoU}}^\mathrm{L}{\mPhi_{\mathrm{PoU}}^\mathrm{L}}^{\top}]_{ij}&=\frac{1}{M}\sum_{k=1}^{M/2}\sigma(\vw_k^{\mathrm{L}}\cdot \vy^{\mathrm{L}}_i)\sigma(\vw_k^{\mathrm{L}}\cdot\vy^{\mathrm{L}}_j), 
\\ [\mPhi_{\mathrm{PoU}}^\mathrm{R}{\mPhi_{\mathrm{PoU}}^\mathrm{R}}^{\top}]_{ij}&=\frac{1}{M}\sum_{k=1}^{M/2}\sigma(\vw_k^\mathrm{R}\cdot\vy^\mathrm{R}_i)\sigma(\vw_k^\mathrm{R}\cdot \vy^\mathrm{R}_j).
\end{aligned}
\label{Eq.Phi with PoU}
\end{equation}

Here $y_j^{\mathrm{L}}=y_j^\mathrm{R}$. We present our first main result, which demonstrates that following the partition of unity, the decomposed matrices $\mPhi_{\mathrm{PoU}}^\mathrm{L}$ and $\mPhi_{\mathrm{PoU}}^\mathrm{R}$ exhibit similar eigenvalues.
\begin{theorem}[Similar eigenvalue distribution]
         Let $\lambda_j^\mathrm{L}$ and $\lambda_j^\mathrm{R}$ be the eigenvalues of $\mPhi_{\mathrm{PoU}}^\mathrm{L}{\mPhi_{\mathrm{PoU}}^\mathrm{L}}^{\top}$   and $\mPhi_{\mathrm{PoU}}^\mathrm{R}{\mPhi_{\mathrm{PoU}}^\mathrm{R}}^{\top}$, respectively. With probability with probability $1-N^2 \E^{-M}$ over the choice of $\{w_k^{\mathrm{L}}, w_k^\mathrm{R}\}_{i=1}^{M/2}$, we have
        \begin{equation*}
            \left(\sum_{j=1}^N\left|\lambda_j^\mathrm{L}-\lambda_j^\mathrm{R}\right|^2\right)^{1 / 2}\le \sqrt{2}\frac{N}{\sqrt{M}}.
        \end{equation*}
\end{theorem}
    
\begin{proof}
    Define 
    \begin{equation*}
        \begin{aligned}
            f(\vw_1^{\mathrm{L}},\ldots,\vw_{M/2}^{\mathrm{L}})&=\frac{1}{M}\sum_{k=1}^{M/2}\sigma(\vw_k^{\mathrm{L}}\cdot \vy^{\mathrm{L}}_i)\sigma(\vw_k^{\mathrm{L}}\cdot \vy^{\mathrm{L}}_j)\\
            g(\vw_1^{\mathrm{R}},\ldots,\vw_{M/2}^{\mathrm{R}})&=\frac{1}{M}\sum_{k=1}^{M/2}\sigma(\vw_k^{\mathrm{R}}\cdot \vy^{\mathrm{R}}_i)\sigma(\vw_k^{\mathrm{R}}\cdot \vy^{\mathrm{R}}_j)
        \end{aligned}
    \end{equation*}
    By the McDiarmid's inequality, for any $\varepsilon>0$, at least $1-2 \E^{-\varepsilon^2M^2}$ over the choice of $\{\vw_k^{\mathrm{L}}\}_{i=1}^{M/2}$ and $\{\vw_k^\mathrm{R}\}_{i=1}^{M/2}$, we have 
\begin{equation*}
\begin{aligned}
            \bigg|{f(\vw_1^{\mathrm{L}},\ldots,\vw_{M/2}^{\mathrm{L}})-\mathbb{E}[f(\vw_1^{\mathrm{L}},\ldots,\vw_{M/2}^{\mathrm{L}})]}\bigg|&\le \varepsilon ,
            \\
            \bigg|{g(\vw_1^{\mathrm{R}},\ldots,\vw_{M/2}^{\mathrm{R}})-\mathbb{E}[g(\vw_1^{\mathrm{R}},\ldots,\vw_{M/2}^{\mathrm{R}})]}\bigg|&\le \varepsilon,
\end{aligned}
\end{equation*}
where the expectation refers to the expectation with respect to $\vw$. By choosing $\varepsilon=\frac{1}{\sqrt{M}}$, and that $\vy_j^{\mathrm{L}}=\vy_j^\mathrm{R}$ for each $j\in \{1,2,\ldots,N/2\}$ hence $\mathbb{E}[f(\vw_1^{\mathrm{L}},\ldots,\vw_{M/2}^{\mathrm{L}})]=\mathbb{E}[g(\vw_1^{\mathrm{R}},\ldots,\vw_{M/2}^{\mathrm{R}})]$ , with probability $1-4 \E^{-M}$ over the choice of $\{\vw_k^{\mathrm{L}}, \vw_k^\mathrm{R}\}_{i=1}^{M}$, we have
\begin{equation}
        \Bigg|\frac{1}{M}\sum_{k=1}^{M/2}\sigma(\vw_k^{\mathrm{L}}\cdot \vy^\mathrm{L}_i)\sigma(\vw_k^{\mathrm{L}}\cdot \vy^{\mathrm{L}}_j)-\frac{1}{M}\sum_{k=1}^{M/2}\sigma(\vw_k^\mathrm{R}\cdot \vy^\mathrm{R}_i)\sigma(\vw_k^\mathrm{R}\cdot \vy^\mathrm{R}_j)\Bigg |\le \frac{2}{\sqrt{M}}.
\end{equation}
This and $\mPhi_{\mathrm{PoU}}^\mathrm{L}{\mPhi_{\mathrm{PoU}}^\mathrm{L}}^{\top}, \mPhi_{\mathrm{PoU}}^\mathrm{R}{\mPhi_{\mathrm{PoU}}^\mathrm{R}}^{\top}\in \sR^{\frac{N}{2}\times \frac{N}{2}}$ together lead to with probability $1-N^2 \E^{-M}$ over the choice of $\{\vw_k^{\mathrm{L}}, \vw_k^\mathrm{R}\}_{i=1}^{M/2}$, we have 
\begin{equation}
        \Big \| \mPhi_{\mathrm{PoU}}^\mathrm{L}{\mPhi_{\mathrm{PoU}}^\mathrm{L}}^{\top}-\mPhi_{\mathrm{PoU}}^\mathrm{R}{\mPhi_{\mathrm{PoU}}^\mathrm{R}}^{\top}\Big \| _\text{F}\le \frac{N}{\sqrt{M} }.
\end{equation}
Hence based on the results in \S0 and \S3 of ~\cite{Kahan1975SpectraON},
\begin{equation}
\left(\sum_{j=1}^N\left|\lambda_j^\mathrm{L}-\lambda_j^\mathrm{R}\right|^2\right)^{1 / 2} \leq \sqrt{2}\Big \|\mPhi_{\mathrm{PoU}}^\mathrm{L}{\mPhi_{\mathrm{PoU}}^\mathrm{L}}^{\top}-\mPhi_{\mathrm{PoU}}^\mathrm{R}{\mPhi_{\mathrm{PoU}}^\mathrm{R}}^{\top}\Big \|_\text{F} \le \sqrt{2}\frac{N}{\sqrt{M}},
\end{equation}
where $\lambda_j^\mathrm{L}$ and $\lambda_j^\mathrm{R}$ are eigenvalues of $\mPhi_{\mathrm{PoU}}^\mathrm{L}{\mPhi_{\mathrm{PoU}}^\mathrm{L}}^{\top}$   and $\mPhi_{\mathrm{PoU}}^\mathrm{R}{\mPhi_{\mathrm{PoU}}^\mathrm{R}}^{\top}$.
And this completes the proof.
\end{proof}

In light of the aforementioned theorem, it is elucidated that upon the completion of the partition of unity, the eigenvalues of $\mPhi_{\mathrm{PoU}}^\mathrm{L}{\mPhi_{\mathrm{PoU}}^\mathrm{L}}^{\top}$   and $\mPhi_{\mathrm{PoU}}^\mathrm{R}{\mPhi_{\mathrm{PoU}}^\mathrm{R}}^{\top}$ exhibit proximity, thereby leading to the following proposition.

\begin{corollary}[Similar effective rank]
 Suppose that the partition of unity is used. The matrices $\mPhi_{\mathrm{PoU}}^\mathrm{L}{\mPhi_{\mathrm{PoU}}^\mathrm{L}}^{\top}$   and $\mPhi_{\mathrm{PoU}}^\mathrm{R}{\mPhi_{\mathrm{PoU}}^\mathrm{R}}^{\top}$ have similar effective rank: 
\begin{equation}
\begin{aligned}
      & 
      & \operatorname{erank}(\mPhi_{\mathrm{PoU}}^\mathrm{L}{\mPhi_{\mathrm{PoU}}^\mathrm{L}}^{\top}) \approx  \operatorname{erank}(\mPhi_{\mathrm{PoU}}^\mathrm{R}{\mPhi_{\mathrm{PoU}}^\mathrm{R}}^{\top}). 
    \end{aligned}
\end{equation}
\end{corollary}

Here, we notice that without partition of unity we can get
\begin{equation}
[\mPhi\mPhi^{\top}]_{ij}=\frac{1}{M}\sum_{k=1}^{M}\sigma(\vw_k\cdot \vx_i)\sigma(\vw_k\cdot \vx_j) \approx \mathbb{E}[\mPhi\mPhi^{\top}]_{ij}.
\label{Eq.Phi without PoU}
\end{equation}

In the context of solving PDEs by NNs, the function $\mPhi$ typically exhibits low-rank behavior, characterized by a few large eigenvalues and a majority of eigenvalues close to zero, as illustrated in Fig.~\ref{Fig.RFM_PoU result}(a).  Thus in this case, if we consider $N$,$M$ large enough, comparing $\mPhi\mPhi^{\top} \in \sR^{N\times N}$ and $\mPhi_{\mathrm{PoU}}{\mPhi_{\mathrm{PoU}}}^{\top} \in \sR^{N\times N}$ from Eq.~\eqref{Eq.Phi with PoU} and Eq.~\eqref{Eq.Phi without PoU} we have 

\begin{equation}
\begin{aligned}
    \operatorname{erank}(\mPhi\mPhi^{\top}) &\approx  \operatorname{erank}(\mPhi_{\mathrm{PoU}}^\mathrm{L}{\mPhi_{\mathrm{PoU}}^\mathrm{L}}^{\top}) \approx  \operatorname{erank}(\mPhi_{\mathrm{PoU}}^\mathrm{R}{\mPhi_{\mathrm{PoU}}^\mathrm{R}}^{\top}).    
\end{aligned}
\end{equation}

From $\mPhi_{\mathrm{PoU}}=\mPhi_{\mathrm{PoU}}^\mathrm{L}\oplus \mPhi_{\mathrm{PoU}}^\mathrm{R}$, we have 

\begin{equation}
\begin{aligned}
\operatorname{erank}(\mPhi_{\mathrm{PoU}}{\mPhi_{\mathrm{PoU}}}^{\top}) &\approx \operatorname{erank}(\mPhi_{\mathrm{PoU}}^\mathrm{L}{\mPhi_{\mathrm{PoU}}^\mathrm{L}}^{\top}) + \operatorname{erank}(\mPhi_{\mathrm{PoU}}^\mathrm{R}{\mPhi_{\mathrm{PoU}}^\mathrm{R}}^{\top}).
\end{aligned}
\end{equation}

We know that in this case 
\begin{equation}
\begin{aligned}
\operatorname{erank}(\mPhi_{\mathrm{PoU}}{\mPhi_{\mathrm{PoU}}}^{\top}) &\approx 2\operatorname{erank}(\mPhi\mPhi^{\top}).
\end{aligned}
\end{equation}

Here, we provide a concise proof for the case of $M_p=2$, which can be readily extended to the case of $M_p \geq 2$, i.e.,

\begin{equation}
\begin{aligned}
\operatorname{erank}(\mPhi_{\mathrm{PoU}}{\mPhi_{\mathrm{PoU}}}^{\top}) &\approx M_p \operatorname{erank}(\mPhi\mPhi^{\top}).
\end{aligned}
\end{equation}

Thus from the above analysis, we can show that as the number of partition of unity $M_p$ increases, the effective rank of the basis function increases. It is important to note that our assumption is based on the premise that $M$ and $N$ are sufficiently large. Considering practical experimental conditions, $M$ and $N$ are generally finite. Therefore, there are limitations on the choice of $M_p$. When $M_p$ becomes excessively large, there might not be enough sample points and NN basis functions within each region, leading to significant approximation and generalization errors within each region. Hence, it is advisable to impose some restrictions on the value of $M_p$ and avoid choosing an excessively large value without limitations.

In the above theoretical analysis, the kernel corresponding to the supervised loss is $\mG=\mG^{[\va]} = \mPhi\mPhi^{\top}$. Fig.~\ref{Fig.RFM_PoU result}(a) shows the eigenvalue distributions of $\mG$ and  Fig.~\ref{Fig.RFM_PoU result}(b) shows the log-log plot for the relationship between the effective rank and the number of PoU $M_p$. The experimental setup in this case is defined as $M_p=2^i, i=1,2,3,4$. The log2-log2 plot exhibits a clear linear relationship, which is consistent with our theoretical analysis results.

\begin{figure}[!ht]
    \centering
    \setcounter {subfigure} 0(a){
    \includegraphics[scale=0.43]{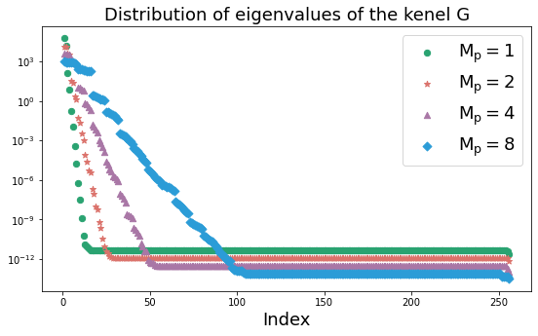}}
    \setcounter {subfigure} 0(b){
    \includegraphics[scale=0.43]{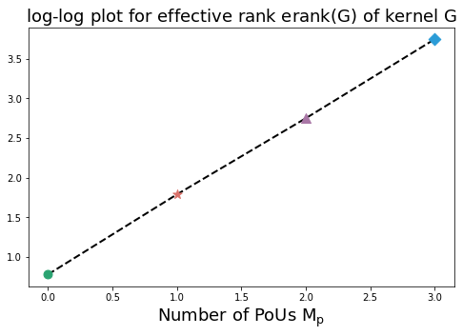}}

    \caption{Number of partitions of unity (PoU) $M_p$. (a) Distribution of singular values of the kernel with different $M_p$'s. (b) Relationship between the effective rank $\operatorname{erank}(\mG)$ and the number of PoU $M_p$. Number of grid points: $N=256$ and number of hidden neurons: $M=1024$. Activation function: $\sigma = \tanh(x)$, and $\{\vw_k\} \sim U(-1, 1)$.}
    \label{Fig.RFM_PoU result}
\end{figure}

\subsubsection{Variance scaling of initialization (VS)}

The initialization of NN parameters is a crucial step in the training process. Here, considering the following NN structure for solving PDEs, i.e.,

\begin{equation}
\phi(\vx;\vtheta) =\sum_{k=1}^M a_k \phi_k(\widetilde{\vx}) = \sum_{j=1}^M a_k \sigma(\vw_k \cdot \widetilde{\vx} + b_k) = \sum_{k=1}^M a_k \sigma(\vw_k \cdot \widetilde{\vx}),
\label{structure}
\end{equation}
where $ \widetilde{\vx} \in [-1,1]^d$ is the normalized input, i.e., $\widetilde{\vx} = \frac{\vx-\vx_0}{r}$ and $\vtheta = \text{vec}(\vtheta_{\text{in}},\vtheta_{\text{out}})$ with $\vtheta_{\text{out}} = \text{vec}(\{a_k\}_{k=1}^{M})$, $\vtheta_{\text{in}} = \text{vec}(\{\vw_k\}_{k=1}^{M})$ are the set of parameters initialized by some distribution. As before, we augment $\tilde{\vx}$  to $(\tilde{\vx},1)^{\top}$ (and still denote it as $\tilde{\vx}$) and let $\vw_k=(\vk_k,b_k)^{\top}$. Here, we mainly focus on the initialization of $\vw_k \sim U(-R_m, R_m)$. We regard $R_m$ as the VS of initialization. In problems involving the use of NNs for solving PDEs, smooth activation functions such as $\sigma(x) = \tanh(x)$ are commonly utilized. Fig.~\ref{Fig.basis_Rm} illustrates the plot of $\tanh(R_m x)$ for various values of $R_m$, where $x$ is in the range of $[-1,1]$. It can be observed that in the function $\tanh(R_m x)$, as $R_m$ increases, the corresponding basis function becomes sharper, resulting in a more localized derivative.

\begin{figure}[!ht]
    \centering
    \setcounter {subfigure} 0(a){
    \includegraphics[scale=0.45]{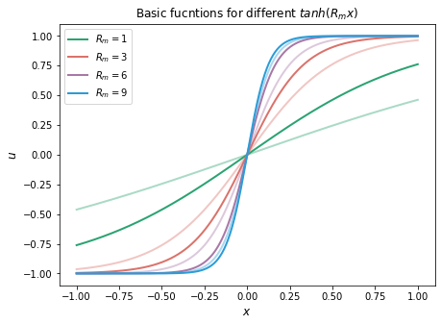}}
    \caption{Basis functions defined on $[-1,1]$, $\{\tanh(R_m x)\}$ for different VS $R_m = \{1,3,6,9\}$.}
    \label{Fig.basis_Rm}
\end{figure}

Similar to the previous case of PoU, here we consider the matrix formed by the NN basis functions 
\begin{equation}
\Phi_{ik}=\sigma(\vw_k\cdot \vx_i) = \tanh(\vw_k\cdot \vx_i), \quad \vw_k \sim U(-R_m, R_m).
\end{equation}

And its corresponding kernel is given by $\mG = \mG^{[\va]} = \mPhi\mPhi^{\top}$.  Fig.~\ref{Fig.RFM_Rm result}(a) shows the eigenvalue distributions of $\mG$ and  Fig.~\ref{Fig.RFM_Rm result}(b) shows the log-log plot for the relationship between the effective rank and the VS $R_m$. The experimental setup in this case is defined as $R_m=2^i, i=1,2,3,4$. The experimental results show that as $R_m$ increases, the effective rank also increases.

\begin{figure}[!ht]
    \centering
    \setcounter {subfigure} 0(a){
    \includegraphics[scale=0.42]{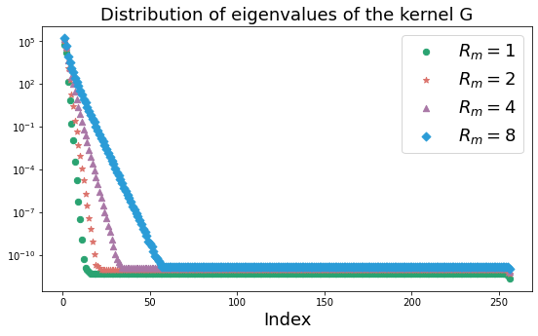}}
    \setcounter {subfigure} 0(b){
    \includegraphics[scale=0.42]{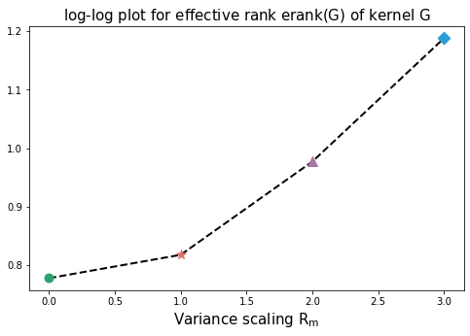}}

    \caption{VS of initialization $R_m$. (a) Distribution of eigenvalues of the kernel with different $R_m$. (b) Relationship between the effective rank and VS of initialization $R_m$. Number of grid points: $N=256$ and number of hidden neurons: $M=1024$. Activation function: $\sigma = \tanh(x)$, and $M_p = 1$.}
    \label{Fig.RFM_Rm result}
\end{figure}

This experimental observation can be readily explained by considering the expansion of the function space associated with the activation function $\sigma(\vw \cdot \vx) = \tanh(\vw \cdot \vx)$, where $\vw \sim U(-R_m, R_m)$. As $R_m$ increases, the range of possible values for $\vw$ expands, leading to a richer and more expressive function space. Consequently, the number of hidden neurons, which can be interpreted as the number of basis functions $\sigma(\vw \cdot \vx)$, also increases. This expansion directly translates into an increase in the effective rank of the matrix formed by the NN basis functions $\Phi_{ik}$. Fig.~\ref{Fig.RFM_Rm effective} visually demonstrates this relationship, showing how larger values of $R_m$ (e.g., $R_m=9$) lead to a larger function space and consequently to higher values of the effective rank. Conversely, with a relatively small $R_m$ value (e.g., $R_m=1$), the function space is constrained. In such case, increasing the number of hidden neurons ($M$) does notably affect the effective rank since the function space remains restricted.

\begin{figure}[!ht]
    \centering
    \setcounter {subfigure} 0(a){
    \includegraphics[scale=0.35]{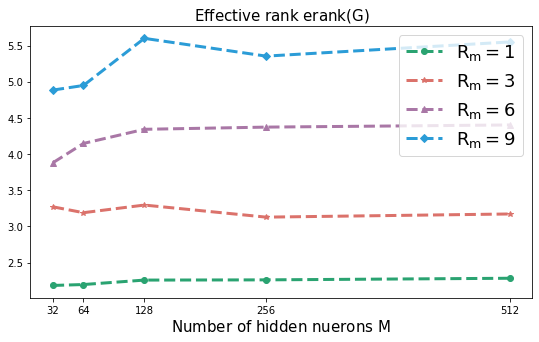}}
    \caption{VS of initialization $R_m$. (a) Relationship between the effective rank $\operatorname{erank}(\mG)$ and number of hidden neurons $M$. The number of grid points $N$ and hidden neurons $M$ are the same, i.e., $M=N$. }
    \label{Fig.RFM_Rm effective}
\end{figure}

From the aforementioned experimental results, as the VS $R_m$ increases, the corresponding effective rank also increases.

\subsection{Effective rank in numerical examples}

The preceding analysis is predicated on the random feature model, where the kernel $\mG=\mG^{[\va]}$ is pre-determined and remains invariant throughout the training process. To validate this analysis, we present two illustrative experiments involving NN models, focusing on regression problems and PINN models for the Helmholtz equation. In both examples, we investigate the relationship between the convergence of the training error and the effective rank of the kernel $\mG$.

\paragraph{\textbf{Regression problem}}

We first consider a regression problem that involves approximating a function $u(\vx)$ using NN functions. We consider a target function,
\begin{equation}
    u(x) = \big(1-\tfrac{1}{2}x^2\big)\cos\big[m(x+\tfrac{1}{2}x^3)\big],
\end{equation}
with $m$ the free parameter related to the frequency of the functions. Fig.~\ref{Fig.NN_Approximation result}(a) illustrates this function for $m=30$. The experimental results, presented in Fig.~\ref{Fig.NN_Approximation result} and Tab.~\ref{tab:Regression}, demonstrate the effectiveness of PoU and VS with increasing values of $R_m$ in enhancing the training dynamics. In this case, the kernel $\mG$ evolves during training. Fig.~\ref{Fig.NN_Approximation result}(b) depicts this evolution for different cases, where lighter points represent the initial state and darker points represent the final state. Interestingly, the eigenvalue distribution of $\mG$ stays relatively consistent across various initialization settings, except for the case when $M_p=1$ and $R_m=1$. Tab.~\ref{tab:Regression} shows that a larger effective rank indicates faster convergence of the training error.

\begin{figure}[!ht]
    \centering
    \setcounter {subfigure} 0(a){
    \includegraphics[scale=0.23]{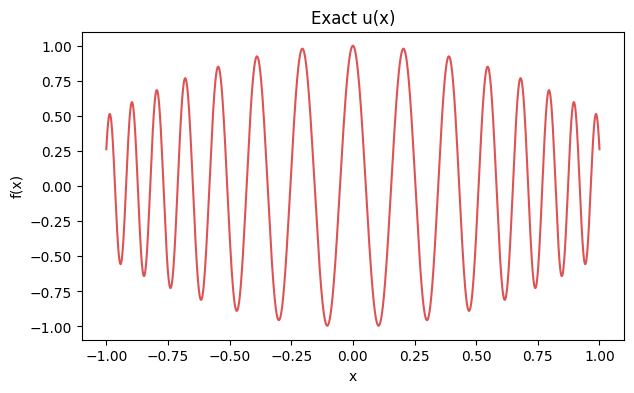}}
    \setcounter {subfigure} 0(b){
    \includegraphics[scale=0.25]{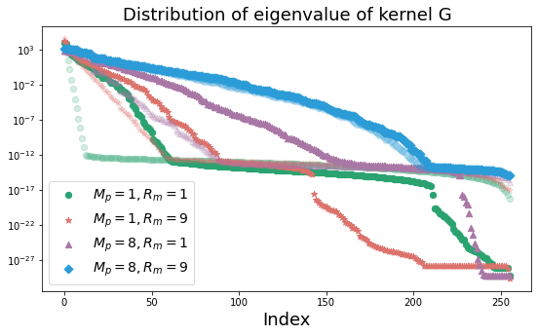}}
    \setcounter {subfigure} 0(c){
    \includegraphics[scale=0.25]{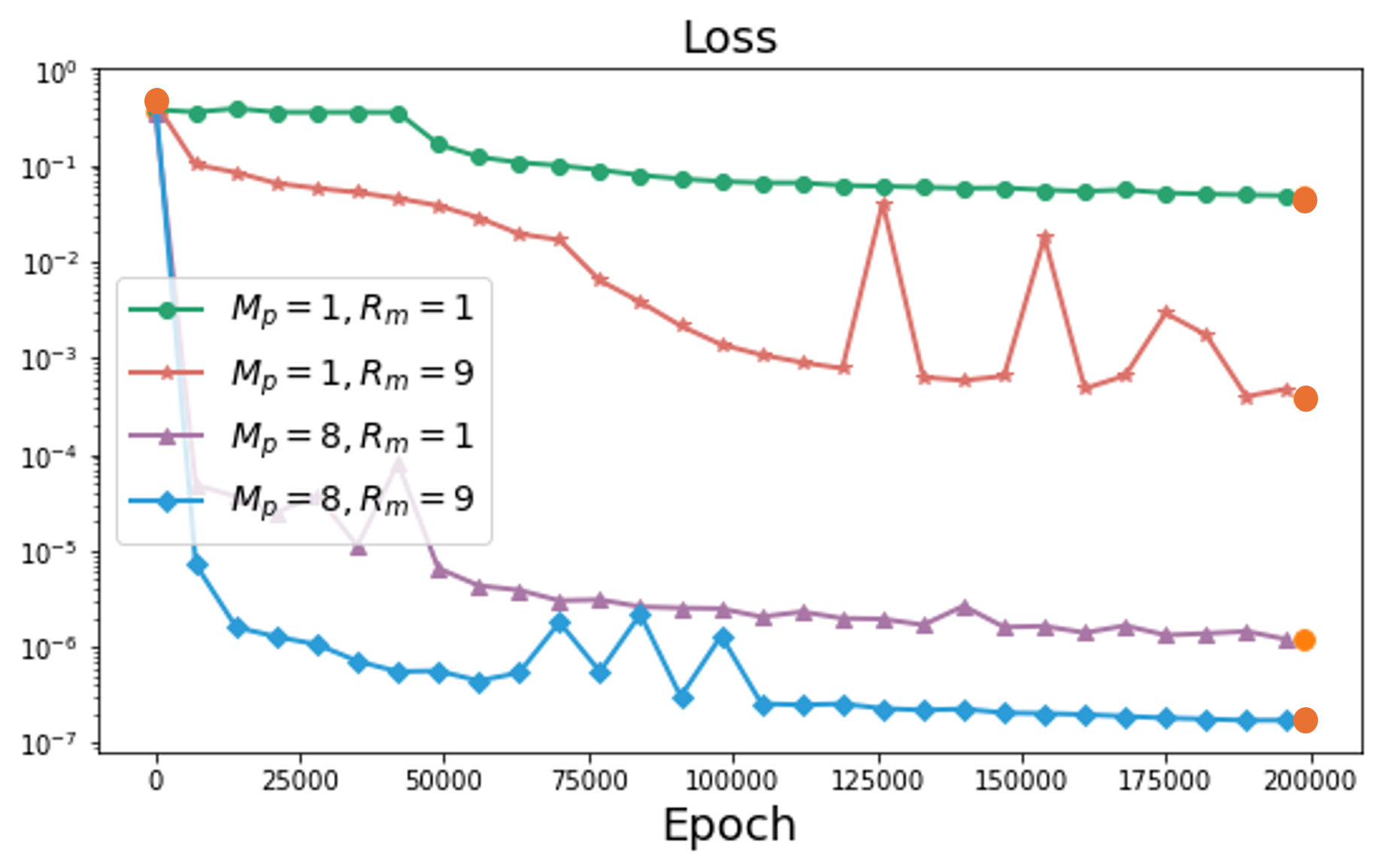}}
    \caption{Regression problem. (a) Exact function $u(x)$. (b) Eigenvalues of kernel $\mG$ at initial and final step of the training process. (c) Loss curve of supervised loss $L$. The orange points denote the specific time instances at which the kernel is computed: one at the initial stage and the other at the final stage of the training process.}
    \label{Fig.NN_Approximation result}
\end{figure}

\begin{table}[!ht]
    \centering
    \begin{tabular}{|c|c|c|c|}
    \hline
            $R_m$ & $M_p$ &Initial $\operatorname{erank}(\mG)$& Final $\operatorname{erank}(\mG)$  \\
    \hline
         1& 1  &1.725&3.036 \\
         9& 1  &  2.214 &3.119 \\
         1& 8  & 13.26&17.30 \\
         9&8 &16.78 &17.99   \\

    \hline
    \end{tabular}
    \caption{Regression: initial and final effective rank $\operatorname{erank}(\mG)$ for different settings for VS of initialization $R_m$ and number of PoU ($M_p$).}
    \label{tab:Regression}
\end{table}

\paragraph{\textbf{1D Helmholtz equation}}
 We then consider using PINN to solve a \textit{Helmholtz equation in one-dimension space} given as follows:
\begin{equation}
\begin{aligned}
 u_{xx}(x) + k^2u(x) &= f(x),\quad  x \in \Omega:=(-1,1), \\
 u(x)&=h(x),\quad x \in \{-1,1\}
\end{aligned}
\end{equation}
where $k=10$. We choose an exact solution to this problem as a multiscale function given by:
\begin{equation}
   u(x) = f_0(x) +c_1x+c_0,
\label{Eq.multiscale}
\end{equation}
where $f_0(x) = \sin(4x)/4 -\sin(8x)/8 + \sin(24x)/36$, $c_1 = (f_0(-1)-f_0(1))/2$ and $c_0 = -(g_0(-1)+g(1))/2$ (as shown in Fig.~\ref{Fig.Helmholtz}(a)). 

The experimental results, presented in Fig.~\ref{Fig.Helmholtz} and Tab.~\ref{tab:Helmholtz}, demonstrate the effectiveness of PoU and VS with increasing values of $R_m$ in enhancing the training dynamics of PINN for solving PDEs. As evident from Tab.~\ref{tab:Helmholtz}, both PoU and VS lead to a significant increase in the effective rank compared to the case without these initialization techniques. This observation is further corroborated by Fig.~\ref{Fig.Helmholtz}(c), which showcases a substantially faster convergence of the loss curve when PoU and VS are employed, particularly with larger values of $R_m$. These results highlight the effectiveness of the effective rank as a metric for quantifying training difficulty and underscore the ability of these initialization techniques to improve training efficiency. This example also highlights the increased optimization difficulty associated with PINN loss compared to supervised loss. As evident from the effective rank values presented in Tab.~\ref{tab:Regression} and Tab.~\ref{tab:Helmholtz}, PINN loss consistently exhibits a lower effective rank, indicating a more challenging training process.

\begin{figure}[!ht]
    \centering
    \setcounter {subfigure} 0(a){
    \includegraphics[scale=0.23]{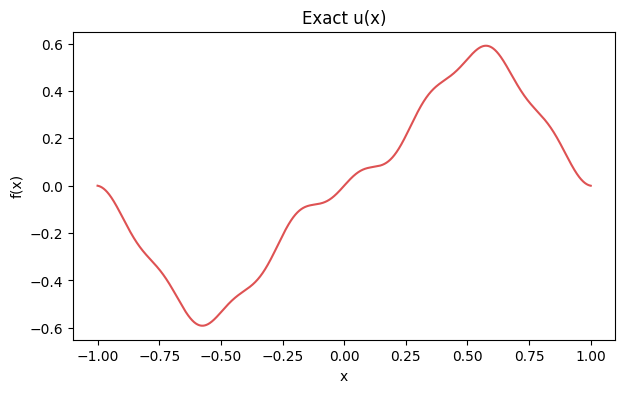}}
    \setcounter {subfigure} 0(b){
    \includegraphics[scale=0.26]{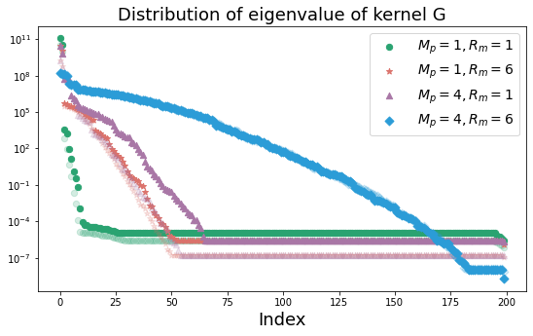}}
    \setcounter {subfigure} 0(c){
    \includegraphics[scale=0.26]{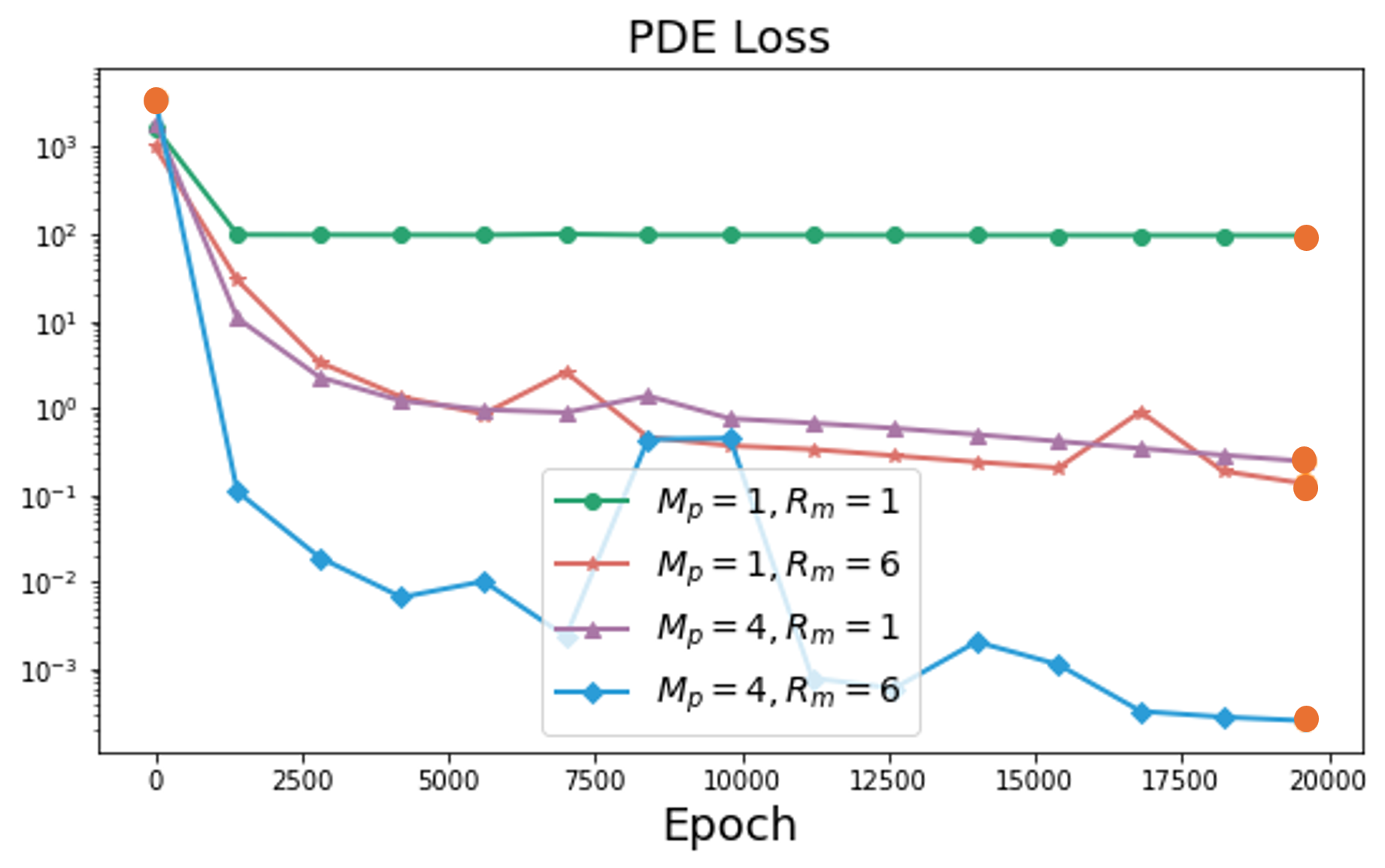}}
    \caption{1D Helmholtz equation. (a) Exact function $u(x)$. (b) Eigenvalues of kernel $\mG$ at initial and final step of the training process. (c) Loss curve of PINN loss $L_{\text{res}}$.  The orange points denote the specific time instances at which the kernel is computed: one at the initial stage and the other at the final stage of the training process. }
    \label{Fig.Helmholtz}
\end{figure}

\begin{table}[!ht]
    \centering
    \begin{tabular}{|c|c|c|c|}
    \hline
            $R_m$ & $M_p$ &Initial $\operatorname{erank}(\mG)$& Final $\operatorname{erank}(\mG)$  \\
    \hline
         1& 1  &1.731&1.693 \\
         6& 1  &  1.768 & 1.816 \\
         1& 4  & 2.112&1.660 \\
         6&  4&11.79 &11.95   \\

    \hline
    \end{tabular}
    \caption{1D Helmholtz equation: initial and final effective rank $\operatorname{erank}(\mG)$ for different settings for VS of initialization $R_m$ and number of PoU ($M_p$).}
    \label{tab:Helmholtz}
\end{table}

\section{Numerical examples on NN based PDE solvers}

In this section, we conduct a series of numerical experiments on three NN-based PDE solvers: PINN~\cite{PINNori}, Deep Ritz~\cite{DeepRitz}, and the operator learning framework DeepOnet~\cite{DeepOnetN}. These experiments involve the application of the initialization techniques, PoU and VS, as previously discussed in the analysis. All code and data accompanying this paper are publicly available at \\ \href{https://github.com/CChenck/Quantify-Training-Difficulty-PoU-and-VS}{https://github.com/CChenck/Quantify-Training-Difficulty-PoU-and-VS}.

\subsection{PINN}

\paragraph{\textbf{Example 4.1.1}} A \textit{Helmholtz equation in two-dimension space} is given as follows:
\begin{equation}
\begin{aligned}
 \Delta u(x, y)+k^2 u(x, y)&=q(x, y), \quad(x, y) \in \Omega:=(-1,1)^2, \\
 u(x, y)&=h(x, y), \quad(x, y) \in \partial \Omega,
\end{aligned}
\end{equation}
where $\Delta$ is the Laplace operator. We choose $u(x, y)=$ $\sin \left(a_1 \pi x\right) \sin \left(a_2 \pi y\right)$, corresponding to a source term of the form
$$
q(x, y)=-\big(\left(a_1 \pi\right)^2+ \left(a_2 \pi\right)^2\big)\sin \left(a_1 \pi x\right) \sin \left(a_2 \pi y\right)+k^2 \sin \left(a_1 \pi x\right) \sin \left(a_2 \pi y\right),
$$
where we take $a_1=1$, $a_2=4$ and $k^2=125$ and the weight $\gamma = 100$.

The results are shown in Fig.~\ref{Fig.PINN2DHelmholtz error1}, Fig.~\ref{Fig.PINN2DHelmholtz error2} and Tab.~\ref{tab:PINN2DHelmholtz}. Fig.~\ref{Fig.PINN2DHelmholtz error2}(b) depicts the training error curve, while Fig.~\ref{Fig.PINN2DHelmholtz error2}(c) showcases the relative $L_2$ error curve. The striking similarity in the downward trends of these curves suggests that both initialization strategies effectively accelerate training, leading to improved accuracy.

\begin{figure}[!ht]
    \centering
    \setcounter {subfigure} 0(a.1){
    \includegraphics[scale=0.19]{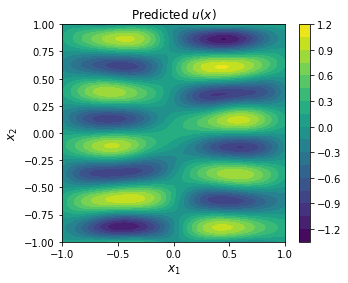}}
    \setcounter {subfigure} 0(b.1){
    \includegraphics[scale=0.19]{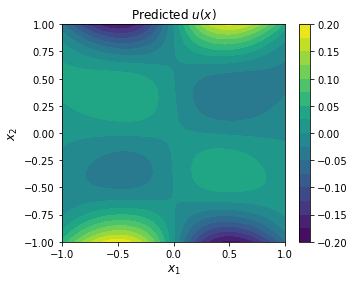}}
    \setcounter {subfigure} 0(c.1){
    \includegraphics[scale=0.19]{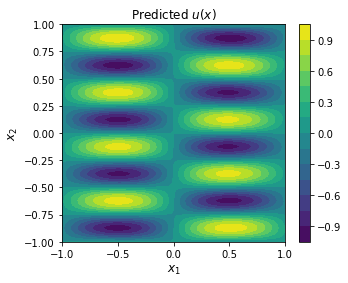}}
    \setcounter {subfigure} 0(d.1){
    \includegraphics[scale=0.19]{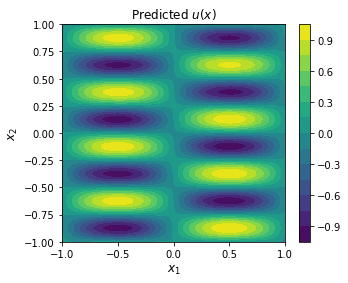}}\\
    \setcounter {subfigure} 0(a.2){
    \includegraphics[scale=0.19]{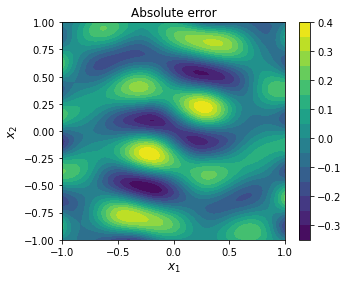}}
    \setcounter {subfigure} 0(b.2){
    \includegraphics[scale=0.19]{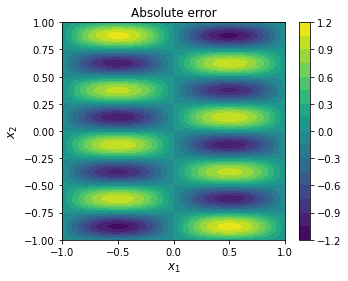}}
    \setcounter {subfigure} 0(c.2){
    \includegraphics[scale=0.18]{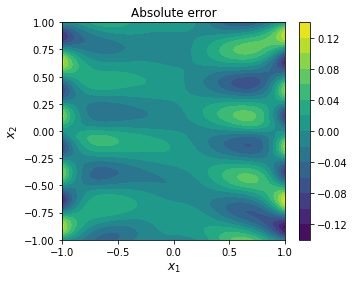}}
    \setcounter {subfigure} 0(d.2){
    \includegraphics[scale=0.18]{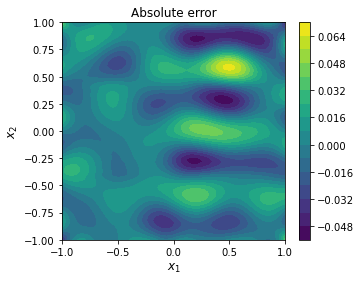}}
    \caption{2D Helmholtz equation: first row: prediction of the PINN models; second row: pointwise absolute error. From left to right, the corresponding situations in the Tab.~\ref{tab:PINN2DHelmholtz} are as follows.}
    \label{Fig.PINN2DHelmholtz error1}
\end{figure}

\begin{figure}[!ht]
    \centering
    \setcounter {subfigure} 0(a){
    \includegraphics[scale=0.25]{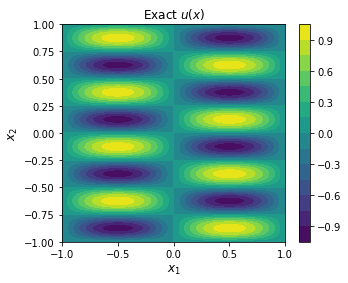}}    
    \setcounter {subfigure} 0(b){
    \includegraphics[scale=0.25]{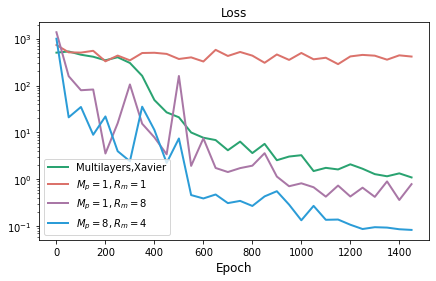}}
    \setcounter {subfigure} 0(c){
    \includegraphics[scale=0.25]{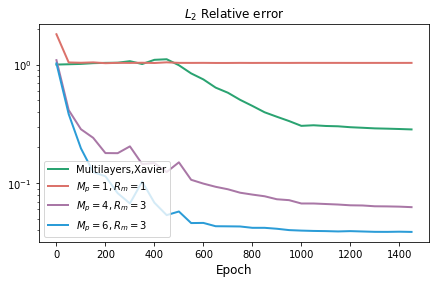}}
    \caption{2D Helmholtz equation: (a) Exact solution. (b) Loss curve of $L$. (c) Loss curve of relative $L_2$ error.}
    \label{Fig.PINN2DHelmholtz error2}
\end{figure}

\begin{table}[!ht]
    \centering
    \begin{tabular}{|c|c|c|c|c|}
    \hline
           $R_m$ & $M_p$ &NN structure& $L_{\infty}$& Rel. $L_2$  \\
    \hline
         Xavier& 1  & $[2,50,50,50,1]$&3.98e-1 & 2.79e-1 \\
         $U(\left[-1, 1\right]^2)$& 1  &  $[2,4800,1]$ &1.11& 1.03 \\
         $U(\left[-3, 3\right]^2)$& 4 $(N_x=1,N_y=4)$ & $[2,4800,1]$&1.31e-1 &6.10e-2\\
         $U(\left[-3, 3\right]^2)$ &6 $(N_x=2,N_y=3)$ &$[2,4800,1]$ &6.64e-2  &3.85e-2 \\

    \hline
    \end{tabular}
    \caption{2D Helmholtz equation: $L_{\infty}$, relative $L_2$ prediction error for different settings for VS of initialization, number of PoU ($M_p$), and NN structure.}
    \label{tab:PINN2DHelmholtz}
\end{table}

\newpage
\paragraph{\textbf{Example 4.1.2}} A steady state \textit{Burger's equation in one-dimension space} is given as follows:
\begin{equation}
\begin{aligned}
 -u{''}(x) + u^{\prime}(x)u(x) &= f(x), \quad x\in (0,8), \\
 u(0) &= c_1,\quad u(8)=c_2.
\end{aligned}
\end{equation}
We choose $u(x)$ as $u(x)=\sin \left(3 \pi x+\frac{3 \pi}{20}\right) \cos \left(2 \pi x+\frac{\pi}{10}\right)+2$.

The results are shown in Fig.~\ref{Fig.1Dburgers} and Tab.~\ref{tab:1Dburgers}. This example demonstrates that both initialization techniques exhibit a similar training acceleration effect for nonlinear PDEs.

\begin{figure}[!ht]
    \centering
    \setcounter {subfigure} 0(a){
    \includegraphics[scale=0.23]{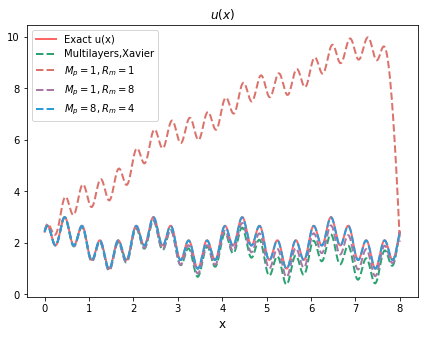}}
    \setcounter {subfigure} 0(b){
    \includegraphics[scale=0.23]{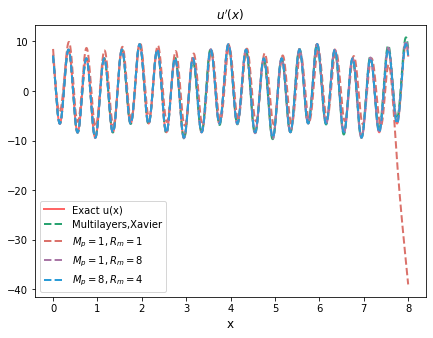}}
    \setcounter {subfigure} 0(c){
    \includegraphics[scale=0.23]{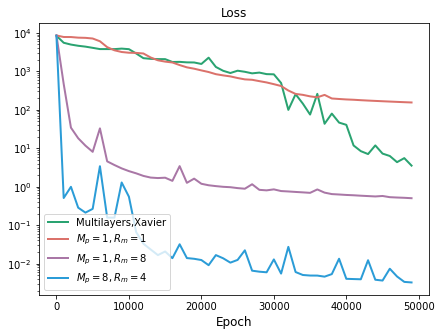}}
    \caption{1D Burger's equation: (a) Exact solution $u(x)$ versus the predictions of PINN models. (b) Derivative of the exact solution $u^{\prime}(x)$ versus the predictions of PINN models. (c) Loss curve of $L$. }
    \label{Fig.1Dburgers}
\end{figure}

\begin{table}[!ht]
    \centering
    \begin{tabular}{|c|c|c|c|c|c|}
    \hline
            $R_m$ & $M_p$ &NN structure& $L_{\infty}$& Rel. $L_2$ & Rel. $H_1$ \\
    \hline
         Xavier& 1  & $[1,50,50,50,1]$&3.88e-1 & 1.09e-1 & 4.86e-2 \\
         $U\left(\left[-1, 1\right]\right)$& 1  &  $[1,512,1]$ &8.28& 2.52 & 1.34 \\
         $U\left(\left[-8, 8\right]\right)$& 1 & $[1,512,1]$&3.87e-1 &  1.09e-1& 4.81e-2\\
         $U\left(\left[-4, 4\right]\right)$ &8 &$[1,512,1]$ &1.11e-3  &3.51e-3 & 4.90e-4\\
    \hline
    \end{tabular}
    \caption{1D Burger's equation: $L_{\infty}$, relative $L_2$ and relative $H_1$ prediction error for different settings for VS of initialization, number of PoU ($M_p$), and NN structure. The activation function used in these models is $\tanh(x)$.}
    \label{tab:1Dburgers}
\end{table}

\newpage
\subsection{Deep Ritz}

The Deep Ritz model~\cite{DeepRitz} is also a deep learning-based approach for solving partial differential equations (PDEs). It formulates the loss function by considering the variation form of the PDEs.

\paragraph{\textbf{Example 4.2.1}} A \textit{second-order elliptic equation in one-dimension space} is given as follows:

\begin{equation}
\begin{aligned}
-u^{\prime \prime}(x)+u(x)&=f(x),  \quad x \in(-1,1), \\
u^{\prime}(-1)&=u^{\prime}(1)=0.
\end{aligned}
\end{equation}

The corresponding variational problem becomes
\begin{equation}
\min _{v \in H^1} \mathcal{J}(v)=\int_{-1}^1\left(\frac{1}{2}\left(v^{\prime}\right)^2+\frac{1}{2} v^2-f v\right) \diff{x}.
\end{equation}

We choose $u(x)$ as Eq.~\eqref{Eq.multiscale}. We employ shallow NNs structure with different activation functions. Then, the Deep Ritz loss reads as 
\begin{equation}
   L =  \int_{-1}^1\left(\frac{1}{2}\left(u_{\vtheta}^{\prime}\right)^2+\frac{1}{2} u_{\vtheta}^2-f u_{\vtheta}\right) \diff{x}.
\end{equation}

The results from the Deep Ritz method are shown in Fig.~\ref{Fig.DeepRitz} and Tab.~\ref{Tab:DeepRitz}.

\begin{figure}[!ht]
    \centering
    \setcounter {subfigure} 0(a){
    \includegraphics[scale=0.23]{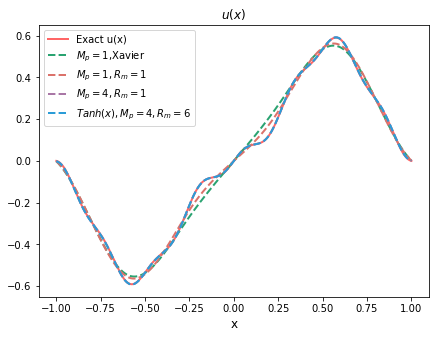}}
    \setcounter {subfigure} 0(b){
    \includegraphics[scale=0.23]{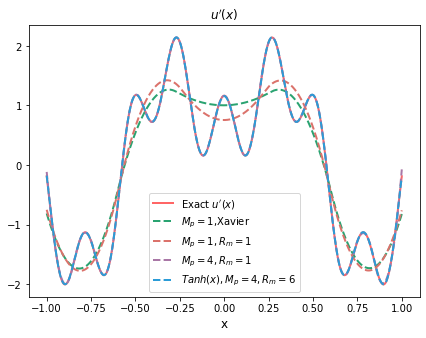}}
    \setcounter {subfigure} 0(c){
    \includegraphics[scale=0.23]{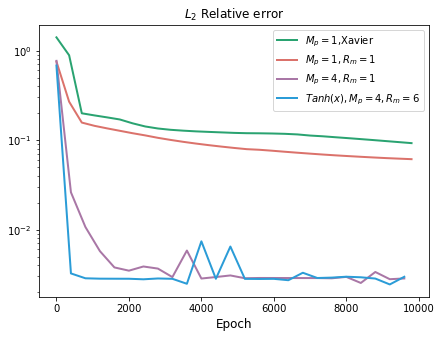}}
    \caption{1D second-order elliptic equation: (a) Exact solution $u(x)$ versus the predictions of PINN models. (b) Derivative of the exact solution $u^{\prime}(x)$ versus the predictions of PINN models. (c) Loss curve of relative $L_2$ error.}
    \label{Fig.DeepRitz}
\end{figure}

\begin{table}[h]
    \centering
    \begin{tabular}{|c|c|c|c|c|c|}
    \hline
            $R_m$ & $M_p$ &Activation Func.& $L_{\infty}$& Rel. $L_2$ & Rel. $H_1$ \\
    \hline
         Xavier& 1  & $\max\{x^3,0\}$&8.49e-02 & 9.11e-02 & 3.64e-01 \\
         $U\left(\left[-1, 1\right]\right)$& 1  & $\max\{x^3,0\}$ &5.19e-02& 6.07e-02 & 3.33e-01\\
         $U\left(\left[-1, 1\right]\right)$& 4 & $\max\{x^3,0\}$&2.38e-03 &  2.88e-03& 7.87e-03\\
         $U\left(\left[-6, 6\right]\right)$ &4 &$\tanh(x)$ &1.85e-03  &2.76e-03 & 2.43e-03\\
    \hline
    \end{tabular}
    \caption{1D second-order elliptic equation: $L_{\infty}$, relative $L_2$ and relative $H_1$ prediction error for different settings for VS of initialization, number of PoU ($M_p$), and activation functions. The NN structure used in these model is $[1,512,1]$.}
    \label{Tab:DeepRitz}
\end{table}

\newpage
\subsection{DeepONet}

Deep operator network (DeepONet)~\cite{DeepOnetN,lu2019deeponet} is a neural operator model specifically designed for operator learning. Operator learning involves the task of learning and approximating operators or mappings between input and output function spaces. In this context, let $\mathcal{V}$ and $\mathcal{U}$ represent the spaces of two functions, $v$ and $u$, respectively, and let $D$ and $D^{\prime}$ denote different domains. The mapping from the input function $v$ to the output function $u$ is represented by an operator $\mathcal{G}$:
$$
\mathcal{G}: \mathcal{V} \ni v \mapsto u \in \mathcal{U} .
$$
In our examples, we aim to approximate $\mathcal{G}$ using DeepONet model $\mathcal{G}_{\vtheta}$ and train the network using a training dataset $\mathcal{T}=$ $\left\{\left(v^{(1)}, u^{(1)}\right),\left(v^{(2)}, u^{(2)}\right), \ldots,\left(v^{(m)}, u^{(m)}\right)\right\}$.

The standard DeepOnet structure is defined as:

\begin{equation}
\mathcal{G}_{\vtheta}(v)(y) = \sum_{k=1}^p \sum_{i=1}^n a_i^k \sigma\left(\sum_{j=1}^m \xi_{i j}^k v\left(x_j\right)+c_i^k\right) \sigma\left(w_k \cdot y+b_k\right).
\end{equation}

Here, $\sigma\left(w_k \cdot y+b_k\right)$ represents the trunk net, which takes the coordinates $y \in D^{\prime}$ as input, and $\sigma\left(\sum_{j=1}^m \xi_{i j}^k u\left(x_j\right)+c_i^k\right)$ represents the branch net, which takes the discretized function $v$ as input. We can interpret the trunk net as the basis functions for solving PDEs mentioned above. By extending the strategy of PoU and VS of initialization to DeepONet, the structure becomes:

\begin{equation}
\mathcal{G}_{\vtheta}(v)(y) = \sum_{k=1}^p \sum_{i=1}^n a_i^k \sigma\left(\sum_{j=1}^m \xi_{i j}^k u\left(x_j\right)+c_i^k\right) {\boldmath{\sum_{h=1}^{M_p}\psi_h(y)}}\sigma\left(w_k \cdot y+b_k\right),
\end{equation}
where $\sum_{h=1}^{M_p}\psi_h(y)$ represents the construction function.

We conduct experiments on two different problems: the 1D Burgers' equation and the 2D Darcy Flow problem. We utilize the dataset generated in ~\cite{li2020FNO,lu2022ComparFNO}. Unless otherwise specified, we use $N = 1000$ training instances and $200$ testing intrances, with supervised learning techniques.


\paragraph{\textbf{Example 4.3.1}} A \textit{Burger's equation in one-dimension space} is given as follows:
$$
\begin{aligned}
\partial_t u(x, t)+\partial_x\left(u^2(x, t) / 2\right) & =\nu \partial_{x x} u(x, t), & & x \in(0,1), t \in(0,1], \\
u(x, 0) & =u_0(x), & & x \in(0,1) ,
\end{aligned}
$$
where the viscosity is set to $v=0.01$. Our objective is to utilize DeepOnet to learn the solution operator that maps the initial condition $u_0(x)$ to $u(x,1)$ of the 1D Burgers' equation, i.e., $\mathcal{G}: u_0(x) \mapsto u(x,1)$.   The initial condition $u_0(x)$ is generated from a Gaussian random field with a Riesz kernel, denoted by $\text{GRF} \sim 
\mathcal{N}\left(0,49^2(-\Delta+49I)^{-4}\right)$ and $\Delta$ and $I$ represent the Laplacian and the identity. We utilize a spatial resolution of $128$ grids to represent both the input and output functions.

The results are shown in Fig.~\ref{Fig.DeepOnetBurger} and Tab.~\ref{Tab:DeepOnetBurger}. From the experimental results, it can be observed that after applying the PoU technique, the decay of the loss function becomes faster. Consequently, within the same number of training epochs, a higher level of accuracy is achieved.
\begin{figure}[!ht]
    \centering
    \setcounter {subfigure} 0(a){
    \includegraphics[scale=0.23]{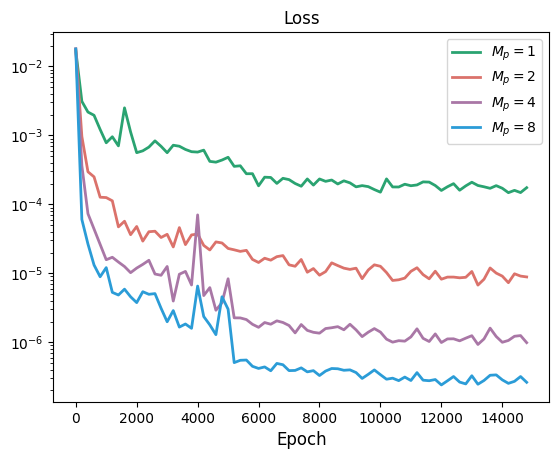}}
    \setcounter {subfigure} 0(b){
    \includegraphics[scale=0.23]{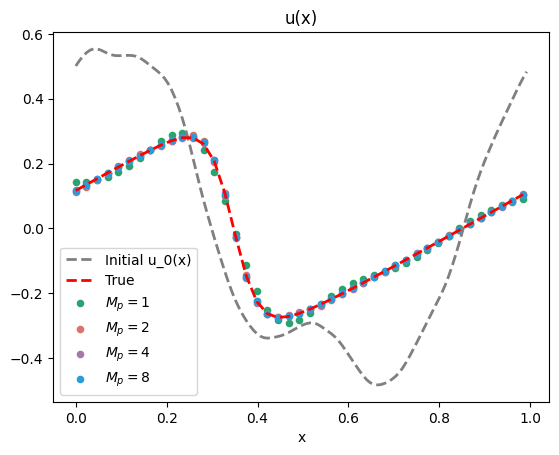}}
    \setcounter {subfigure} 0(c){
    \includegraphics[scale=0.23]{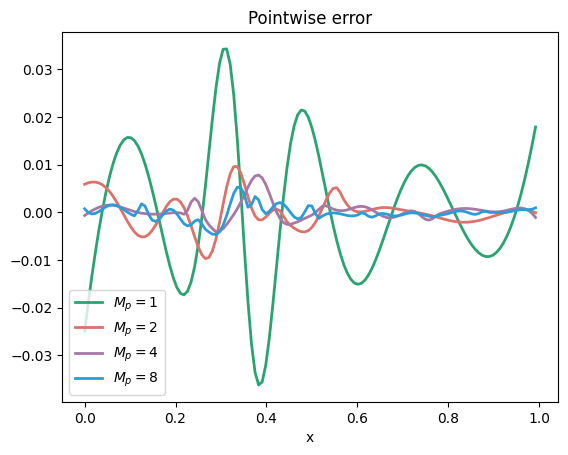}}
    \caption{Operator learning: 1D Burger's equation: (a) Loss curve of the supervised loss $L$. (b) An example of input function $u_0(x)$ and the exact solution $u(x,1)$ versus the prediction of DeepONet models. (c) Pointwise error for the example.}
    \label{Fig.DeepOnetBurger}
\end{figure}

\begin{table}[h]
    \centering
    \begin{tabular}{|c|c|c|c|c|}
    \hline
            $M_p$  & Rel. $L_{\infty}$ &Rel. $L_{2}$ & Rel. $L_{\infty}$ &Rel. $L_{2}$ \\
            &  128 grids & 128 grids &  512 grids  &512 grids \\  
    \hline
         1&  3.72e-01
 &1.02e-01 & 4.10e-01 & 1.03e-01\\
 
         2& 2.22e-01  &  2.96e-02 &2.25e-01& 2.97e-02 \\
         
         4& 8.88e-02 & 1.42e-02&8.95e-02 &  1.43e-02\\
         
         8 &8.09e-02 &1.13e-02 &8.10e-02 &1.14e-02\\

    \hline
    \end{tabular}
    \caption{Operator learning: 1D Burger's equation: relative $L_\infty$ and relative $L_2$ prediction error for different settings for the number of Points of Uncertainty (PoU) ($M_p$). The structure of the trunk network is $[1,128,128,128]$. The training grid size is $128$. The testing grid size is $512$.}
    \label{Tab:DeepOnetBurger}
\end{table}

\paragraph{\textbf{Example 4.3.2}} A \textit{Darcy flows in two-dimension space} is given as follows:

$$
\begin{aligned}
-\nabla \cdot(k(x,y) \nabla u(x,y)) & =f(x,y), & & (x,y) \in(0,1)^2 \\
u(x,y) & =0, & & (x,y) \in \partial(0,1)^2,
\end{aligned}
$$
where $a$ represents the permeability field, $u$ denotes the pressure field, and $f$ represents the source term. Our objective is to employ DeepOnet to approximate the operator that maps $k(x,y)$ to $u(x,y)$, i.e., $\mathcal{G}: k(x,y) \mapsto u(x,y)$. The coefficient field $k$ is defined on $\psi(\mu)$, where $\mu \sim \mathcal{N}\left(0,(-\triangle+9I)^{-2}\right)$. The mapping function $\psi$ binarizes the function, converting positive values to 12 and negative values to 3. The grid resolution for both $a$ and $u$ is $85 \times 85$.

The results are shown in Fig.~\ref{Fig.DeepOnetDarcy_2}, Fig.~\ref{Fig.DeepOnetDarcy_1}, and Tab.~\ref{Tab:DeepOnetDarcy}. The experimental results once again demonstrate that the PoU strategy can accelerate the convergence rate of the loss function. Furthermore, by incorporating PoU, a simple shallow NN can outperform a deep NN. Therefore, it can be observed that utilizing PoU enables us to achieve better results with a NN that requires fewer parameters.

\begin{figure}[!ht]
    \centering
    \setcounter {subfigure} 0(a){
    \includegraphics[scale=0.27]{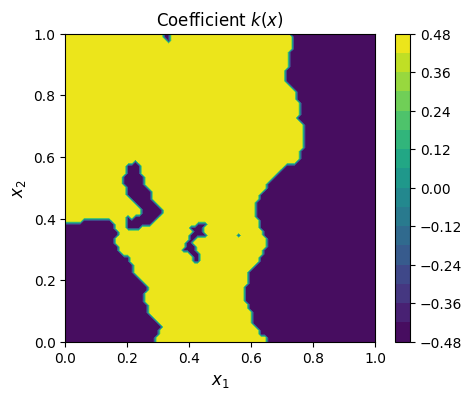}}    
    \setcounter {subfigure} 0(b){
    \includegraphics[scale=0.27]{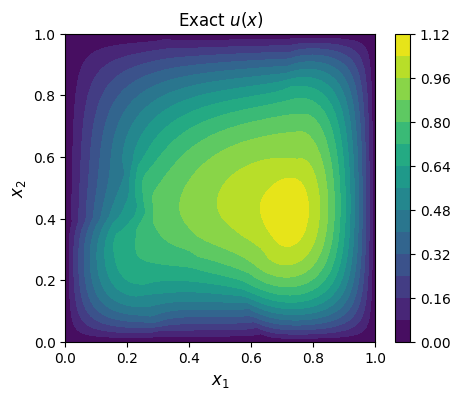}}
    \setcounter {subfigure} 0(c){
    \includegraphics[scale=0.27]{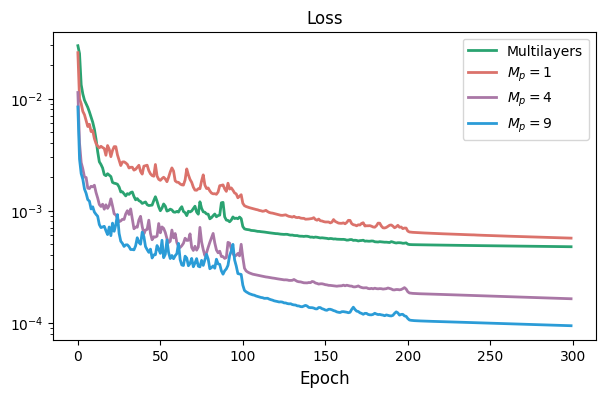}}
    \caption{Operator learning: 2D Darcy flow equation: (a) An example of input coefficient function $k(x_1,x_2)$. (b) The corresponding exact solution $u(x_1,x_2)$. (c) Loss curve of the supervised loss $L$.}
    \label{Fig.DeepOnetDarcy_2}
\end{figure}

\begin{figure}[!ht]
    \setcounter {subfigure} 0(a.1){
    \includegraphics[scale=0.19]{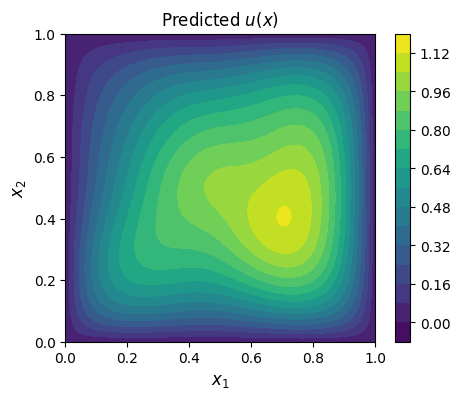}}
    \setcounter {subfigure} 0(b.1){
    \includegraphics[scale=0.19]{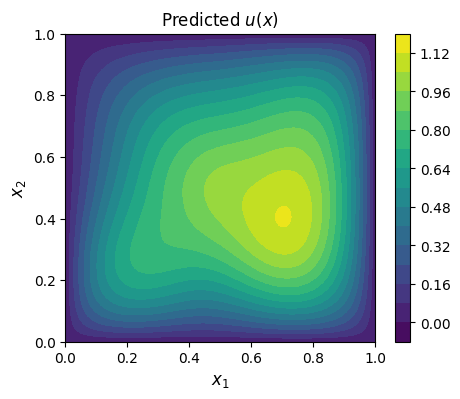}}
    \setcounter {subfigure} 0(c.1){
    \includegraphics[scale=0.19]{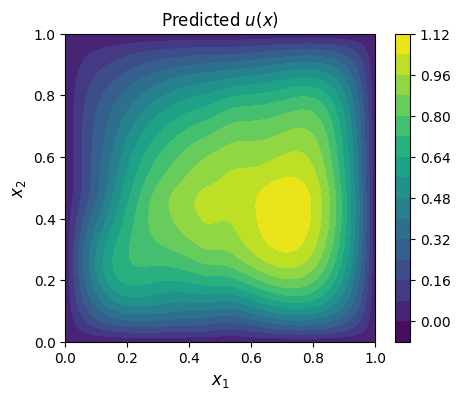}}
    \setcounter {subfigure} 0(d.1){
    \includegraphics[scale=0.19]{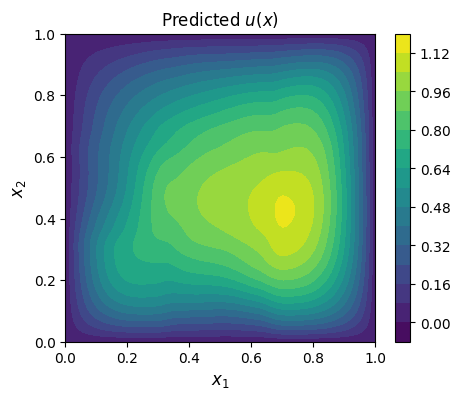}}\\
    \setcounter {subfigure} 0(a.2){
    \includegraphics[scale=0.18]{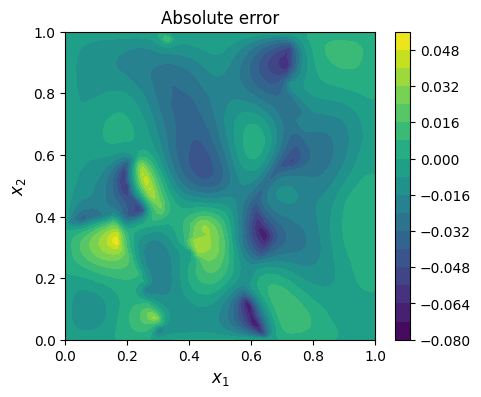}}
    \setcounter {subfigure} 0(b.2){
    \includegraphics[scale=0.18]{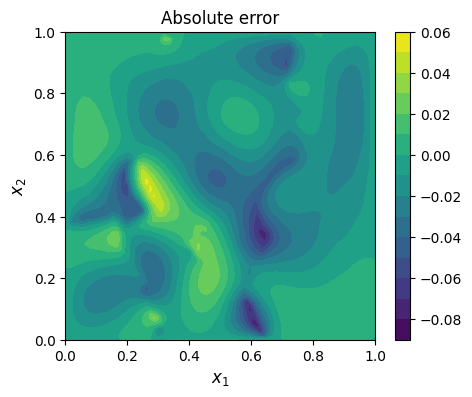}}
    \setcounter {subfigure} 0(c.2){
    \includegraphics[scale=0.18]{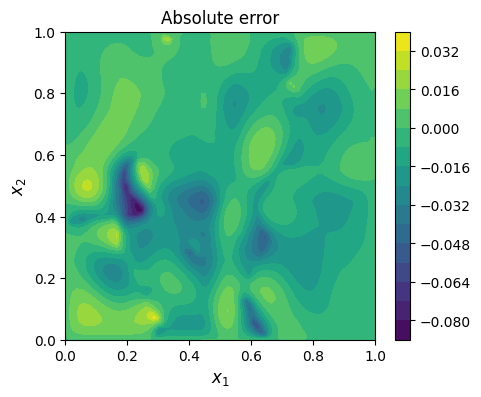}}
    \setcounter {subfigure} 0(d.2){
    \includegraphics[scale=0.18]{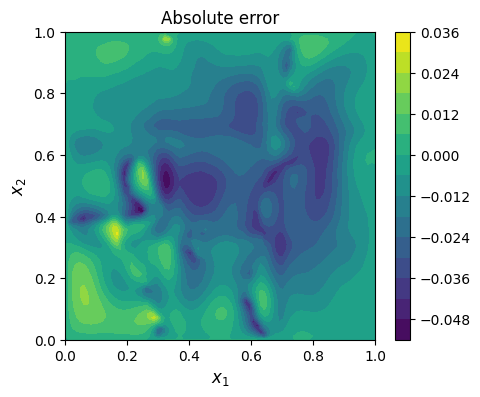}}
    \caption{Operator learning: 2D Darcy flow equation: first row: prediction of the PINN models. Second row: pointwise absolute error. From left to right, the corresponding situations in the Tab.~\ref{Tab:DeepOnetDarcy} are as follows.}
    \label{Fig.DeepOnetDarcy_1}
\end{figure}

\begin{table}[!ht]
    \centering
    \begin{tabular}{|c|c|c|c|}
    \hline
            $M_p$  & Trunk net structure&Rel. $L_{\infty}$& Rel. $L_{2}$  \\
    \hline
         1&  [1,128,128,128,128]
 &3.10e-02 & 3.58e-02 \\
 
         1& [1,128]  &  3.40e-02 &3.88e-01 \\
         
         4& [1,128]  & 2.34e-02&2.72e-02 \\
         
         9 &[1,128]  &2.26e-02 &2.60e-02 \\

    \hline
    \end{tabular}
    \caption{Operator learning: 2D Darcy flow equation: relative $L_\infty$ and relative $L_2$ prediction error for different settings for the number of PoU ($M_p$) and trunk net structures.}
    \label{Tab:DeepOnetDarcy}
\end{table}

\newpage
\section{Conclusion and discussion}

The overall framework of the paper is illustrated in Fig.~\ref{Fig.framwork}.

\begin{figure}[!hbtp]
\centering
\includegraphics[width=1\textwidth]{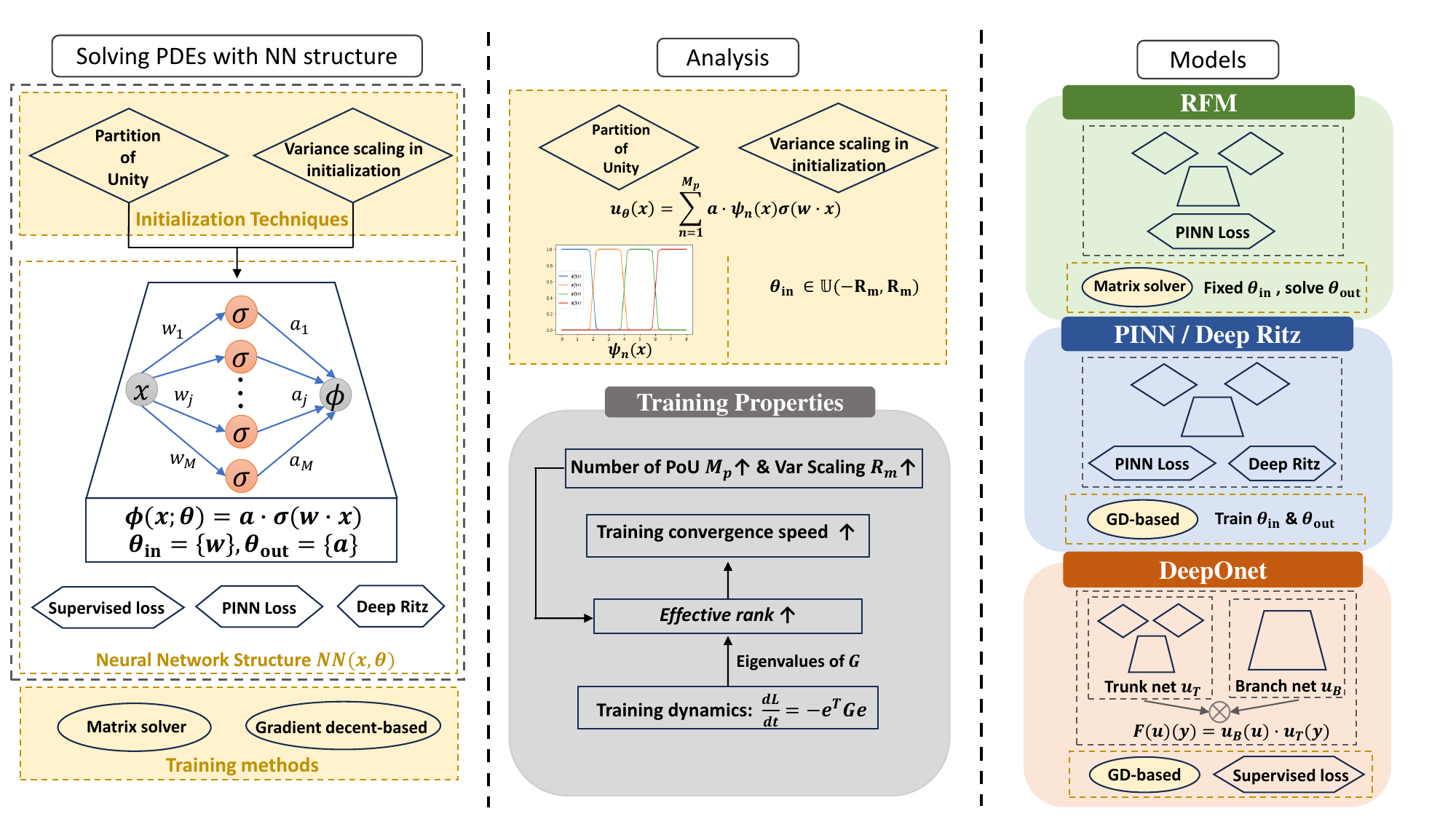}
\caption{The framework of the paper.}
\label{Fig.framwork}
\end{figure}

This paper investigates the training dynamics of NN-based methods for solving partial differential equations (PDEs). We first analyze the training dynamics and discover that the eigenvalue distribution of the kernel $\mG$ influences the convergence rate of the training error. Based on this finding, we apply the concept of effective rank $\operatorname{erank}(\mG)$~\cite{roy2007effective} as a metric for quantifying training difficulty. Our theoretical and experimental results demonstrate that a larger effective rank leads to faster convergence of the training error. We then focus on two initialization techniques used in random feature models~\cite{chen2022bridging} for solving PDEs: PoU and VS. For two-layer NNs, we theoretically establish a relationship between the number of PoU ($M_p$) and the effective rank, $\operatorname{erank}(\mPhi_{\mathrm{PoU}}) \approx M_p \operatorname{erank}(\mPhi)$ assuming a sufficiently large number of grid points ($N$) and hidden neurons ($M$). Our experimental observations corroborate this theoretical finding. We further demonstrate through experiments that increasing the VS ($R_m$) also leads to a larger effective rank. These findings are validated in the context of NNs, confirming the effectiveness of the effective rank metric and the two initialization techniques. Finally, we conduct extensive numerical experiments to demonstrate that both PoU and VS accelerate training in various NN-based PDE solvers, including PINN, Deep Ritz, and even DeepOnet for operator learning.

In addition to the initialization techniques explored in this work, several other aspects of NN-based PDE solvers deserve further investigation. One promising avenue for future research is to explore how the structure of the NN itself influences the effective rank and, consequently, the training dynamics. Examining the interplay between network architecture and training efficiency could lead to the development of even more effective NN designs for solving PDEs.

Another important consideration is the balance between the residual loss $L_{\text{res}}$ and the boundary condition loss $L_{\text{bc}}$ during training. Our observations suggest that the effective rank associated with the residual loss differs significantly from that of the supervised loss commonly used for boundary conditions. This finding raises the potential for designing the balance parameter $\gamma$ based on the effective rank, potentially leading to improved training convergence and accuracy.

Furthermore, our analysis focuses on gradient-based training methods. We observed that these methods can lead to slow or incomplete convergence of small eigenvalues in the kernel $\mG$, which may contribute to the accuracy gap between NN-based and traditional PDE solvers. In contrast, random feature methods for solving PDEs, where the parameters of the nonlinear layer are fixed and a direct matrix solver is employed, tend to achieve higher accuracy. The observation that the eigenvalue distribution of the kernel $\mG$ stabilizes during the later stages of training suggests the potential for developing hybrid solvers that combine gradient-based and direct matrix-based approaches. This could offer a promising avenue for achieving both high accuracy and computational efficiency in NN-based methods for PDE solution.

\section*{Acknowledgements}
This work of T.L. is sponsored by the National Key R\&D Program of China Grant No. 2022YFA1008200 (T. L.), the National Natural Science Foundation of China Grant No. 12101401 (T. L.),  Shanghai Municipal Science and Technology Key Project No. 22JC1401500 (T. L.). The work of Y.X. was supported by the Project of Hetao Shenzhen-HKUST Innovation Cooperation Zone HZQB-KCZYB-2020083.

\bibliographystyle{siamplain}
\bibliography{extracted.bib}

\begin{thebibliography}{10}

\bibitem{nlp2}
{\sc J.~Achiam, S.~Adler, S.~Agarwal, L.~Ahmad, I.~Akkaya, F.~L. Aleman, D.~Almeida, J.~Altenschmidt, S.~Altman, S.~Anadkat, et~al.}, {\em G{P}{T}-4 technical report}, arXiv preprint arXiv:2303.08774,  (2023).

\bibitem{spetralbias3}
{\sc R.~Basri, M.~Galun, A.~Geifman, D.~Jacobs, Y.~Kasten, and S.~Kritchman}, {\em Frequency bias in neural networks for input of non-uniform density}, in International Conference on Machine Learning, PMLR, 2020, pp.~685--694.

\bibitem{nlp1}
{\sc T.~Brown, B.~Mann, N.~Ryder, M.~Subbiah, J.~D. Kaplan, P.~Dhariwal, A.~Neelakantan, P.~Shyam, G.~Sastry, A.~Askell, et~al.}, {\em Language models are few-shot learners}, Advances in Neural Information Processing Systems, 33 (2020), pp.~1877--1901.

\bibitem{chen2024automatic}
{\sc C.~Chen, Y.~Yang, Y.~Xiang, and W.~Hao}, {\em Automatic differentiation is essential in training neural networks for solving differential equations}, arXiv preprint arXiv:2405.14099,  (2024).

\bibitem{chen2022bridging}
{\sc J.~Chen, X.~Chi, Z.~Yang, and W.~E}, {\em Bridging traditional and machine learning-based algorithms for solving {P}{D}{E}s: the random feature method}, Journal of Machine Learning, 1 (2022), pp.~268--98.

\bibitem{rfm4}
{\sc J.~Chen, W.~E, and Y.~Sun}, {\em Optimization of random feature method in the high-precision regime}, Communications on Applied Mathematics and Computation, 6 (2024), pp.~1490--1517.

\bibitem{rfm2}
{\sc J.~Chen, Y.~Luo, and W.~E}, {\em The random feature method for time-dependent problems}, arXiv preprint arXiv:2304.06913,  (2023).

\bibitem{rfm1}
{\sc X.~Chi, J.~Chen, and Z.~Yang}, {\em The random feature method for solving interface problems}, Computer Methods in Applied Mechanics and Engineering, 420 (2024), p.~116719.

\bibitem{QRmethod}
{\sc W.~Gander}, {\em Algorithms for the {Q}{R} decomposition}, Res. Rep, 80 (1980), pp.~1251--1268.

\bibitem{optimiz3}
{\sc W.~Hao, Q.~Hong, and X.~Jin}, {\em Gauss newton method for solving variational problems of {P}{D}{E}s with neural network discretizaitons}, Journal of Scientific Computing, 100 (2024), p.~17.

\bibitem{universal3}
{\sc K.~Hornik, M.~Stinchcombe, and H.~White}, {\em Multilayer feedforward networks are universal approximators}, Neural networks, 2 (1989), pp.~359--366.

\bibitem{universal4}
{\sc K.~Hornik, M.~Stinchcombe, and H.~White}, {\em Universal approximation of an unknown mapping and its derivatives using multilayer feedforward networks}, Neural Networks, 3 (1990), pp.~551--560.

\bibitem{train2}
{\sc A.~A. Howard, S.~H. Murphy, S.~E. Ahmed, and P.~Stinis}, {\em {S}tacked networks improve physics-informed training: applications to neural networks and deep operator networks}, arXiv preprint arXiv:2311.06483,  (2023).

\bibitem{structure2}
{\sc A.~D. Jagtap and G.~E. Karniadakis}, {\em Extended physics-informed neural networks ({X}{P}{I}{N}{N}s): A generalized space-time domain decomposition based deep learning framework for nonlinear partial differential equations}, Communications in Computational Physics, 28 (2020).

\bibitem{ai4sci2}
{\sc J.~Jumper, R.~Evans, A.~Pritzel, T.~Green, M.~Figurnov, O.~Ronneberger, K.~Tunyasuvunakool, R.~Bates, A.~{\v{Z}}{\'\i}dek, A.~Potapenko, et~al.}, {\em Highly accurate protein structure prediction with {A}lpha{F}old}, Nature, 596 (2021), pp.~583--589.

\bibitem{Kahan1975SpectraON}
{\sc W.~Kahan}, {\em Spectra of nearly hermitian matrices}, Proceedings of the American Mathematical Society, 48 (1975), pp.~11--17.

\bibitem{failurepinn1}
{\sc A.~Krishnapriyan, A.~Gholami, S.~Zhe, R.~Kirby, and M.~W. Mahoney}, {\em Characterizing possible failure modes in physics-informed neural networks}, Advances in Neural Information Processing Systems, 34 (2021), pp.~26548--26560.

\bibitem{lan2023dosnet}
{\sc Y.~Lan, Z.~Li, J.~Sun, and Y.~Xiang}, {\em {D}{O}{S}net as a non-black-box {P}{D}{E} solver: When deep learning meets operator splitting}, Journal of Computational Physics, 491 (2023), p.~112343.

\bibitem{li2020FNO}
{\sc Z.~Li, N.~Kovachki, K.~Azizzadenesheli, B.~Liu, K.~Bhattacharya, A.~Stuart, and A.~Anandkumar}, {\em Fourier neural operator for parametric partial differential equations}, arXiv preprint arXiv:2010.08895,  (2020).

\bibitem{init1}
{\sc X.~Liu, X.~Zhang, W.~Peng, W.~Zhou, and W.~Yao}, {\em A novel meta-learning initialization method for physics-informed neural networks}, Neural Computing and Applications, 34 (2022), pp.~14511--14534.

\bibitem{structure3}
{\sc Z.~Liu, W.~Cai, and Z.-Q.~J. Xu}, {\em Multi-scale deep neural network ({M}scale{D}{N}{N}) for solving poisson-boltzmann equation in complex domains}, arXiv preprint arXiv:2007.11207,  (2020).

\bibitem{optimiz2}
{\sc B.~Lu, C.~Moya, and G.~Lin}, {\em {N}{S}{G}{A}-{P}{I}{N}{N}: a multi-objective optimization method for physics-informed neural network training}, Algorithms, 16 (2023), p.~194.

\bibitem{lu2019deeponet}
{\sc L.~Lu, P.~Jin, and G.~E. Karniadakis}, {\em Deeponet: Learning nonlinear operators for identifying differential equations based on the universal approximation theorem of operators}, arXiv preprint arXiv:1910.03193,  (2019).

\bibitem{DeepOnetN}
{\sc L.~Lu, P.~Jin, G.~Pang, Z.~Zhang, and G.~E. Karniadakis}, {\em Learning nonlinear operators via deeponet based on the universal approximation theorem of operators}, Nature Machine Intelligence, 3 (2021), pp.~218--229.

\bibitem{lu2022ComparFNO}
{\sc L.~Lu, X.~Meng, S.~Cai, Z.~Mao, S.~Goswami, Z.~Zhang, and G.~E. Karniadakis}, {\em A comprehensive and fair comparison of two neural operators (with practical extensions) based on fair data}, Computer Methods in Applied Mechanics and Engineering, 393 (2022), p.~114778.

\bibitem{universal1}
{\sc Y.~Lu and J.~Lu}, {\em A universal approximation theorem of deep neural networks for expressing probability distributions}, Advances in Neural Information Processing Systems, 33 (2020), pp.~3094--3105.

\bibitem{leastsquare}
{\sc S.~J. Miller}, {\em The method of least squares}, Mathematics Department Brown University, 8 (2006), pp.~5--11.

\bibitem{pou2}
{\sc B.~Moseley, A.~Markham, and T.~Nissen-Meyer}, {\em Finite basis physics-informed neural networks ({F}{B}{P}{I}{N}{N}s): a scalable domain decomposition approach for solving differential equations}, Advances in Computational Mathematics, 49 (2023), p.~62.

\bibitem{optimiz1}
{\sc J.~M{\"u}ller and M.~Zeinhofer}, {\em Achieving high accuracy with {P}{I}{N}{N}s via energy natural gradient descent}, in International Conference on Machine Learning, PMLR, 2023, pp.~25471--25485.

\bibitem{spetralbias1}
{\sc N.~Rahaman, A.~Baratin, D.~Arpit, F.~Draxler, M.~Lin, F.~Hamprecht, Y.~Bengio, and A.~Courville}, {\em On the spectral bias of neural networks}, in International Conference on Machine Learning, PMLR, 2019, pp.~5301--5310.

\bibitem{PINNori}
{\sc M.~Raissi, P.~Perdikaris, and G.~E. Karniadakis}, {\em Physics-informed neural networks: A deep learning framework for solving forward and inverse problems involving nonlinear partial differential equations}, Journal of Computational physics, 378 (2019), pp.~686--707.

\bibitem{spetralbias2}
{\sc B.~Ronen, D.~Jacobs, Y.~Kasten, and S.~Kritchman}, {\em The convergence rate of neural networks for learned functions of different frequencies}, Advances in Neural Information Processing Systems, 32 (2019).

\bibitem{roy2007effective}
{\sc O.~Roy and M.~Vetterli}, {\em The effective rank: A measure of effective dimensionality}, in 2007 15th European signal processing conference, IEEE, 2007, pp.~606--610.

\bibitem{rnn1}
{\sc Y.~Shang, F.~Wang, and J.~Sun}, {\em Randomized neural network with {P}etrov--{G}alerkin methods for solving linear and nonlinear partial differential equations}, Communications in Nonlinear Science and Numerical Simulation, 127 (2023), p.~107518.

\bibitem{pou3}
{\sc K.~Shukla, A.~D. Jagtap, and G.~E. Karniadakis}, {\em Parallel physics-informed neural networks via domain decomposition}, Journal of Computational Physics, 447 (2021), p.~110683.

\bibitem{init2}
{\sc C.~Si and M.~Yan}, {\em Initialization-enhanced physics-informed neural network with domain decomposition ({I}{D}{P}{I}{N}{N})}, arXiv preprint arXiv:2406.03172,  (2024).

\bibitem{rnn2}
{\sc J.~Sun, S.~Dong, and F.~Wang}, {\em Local randomized neural networks with discontinuous {G}alerkin methods for partial differential equations}, Journal of Computational and Applied Mathematics, 445 (2024), p.~115830.

\bibitem{pou1}
{\sc N.~Trask, A.~Henriksen, C.~Martinez, and E.~Cyr}, {\em Hierarchical partition of unity networks: fast multilevel training}, in Mathematical and Scientific Machine Learning, PMLR, 2022, pp.~271--286.

\bibitem{cv2}
{\sc A.~Vaswani, N.~Shazeer, N.~Parmar, J.~Uszkoreit, L.~Jones, A.~N. Gomez, {\L}.~Kaiser, and I.~Polosukhin}, {\em Attention is all you need}, Advances in Neural Information Processing Systems, 30 (2017).

\bibitem{failurepinn2}
{\sc S.~Wang, Y.~Teng, and P.~Perdikaris}, {\em Understanding and mitigating gradient flow pathologies in physics-informed neural networks}, SIAM Journal on Scientific Computing, 43 (2021), pp.~A3055--A3081.

\bibitem{loss2}
{\sc S.~Wang, X.~Yu, and P.~Perdikaris}, {\em When and why {P}{I}{N}{N}s fail to train: A neural tangent kernel perspective}, Journal of Computational Physics, 449 (2022), p.~110768.

\bibitem{structure1}
{\sc Y.~Wang and C.-Y. Lai}, {\em Multi-stage neural networks: Function approximator of machine precision}, Journal of Computational Physics, 504 (2024), p.~112865.

\bibitem{spetralbias4}
{\sc Z.-Q.~J. Xu, Y.~Zhang, T.~Luo, Y.~Xiao, and Z.~Ma}, {\em Frequency principle: {F}ourier analysis sheds light on deep neural networks}, arXiv preprint arXiv:1901.06523,  (2019).

\bibitem{cv1}
{\sc L.~Yang, Z.~Zhang, Y.~Song, S.~Hong, R.~Xu, Y.~Zhao, W.~Zhang, B.~Cui, and M.-H. Yang}, {\em Diffusion models: A comprehensive survey of methods and applications}, ACM Computing Surveys, 56 (2023), pp.~1--39.

\bibitem{strategies4}
{\sc Y.~Yang, Q.~Chen, and W.~Hao}, {\em Homotopy relaxation training algorithms for infinite-width two-layer {R}e{L}u neural networks}, arXiv preprint arXiv:2309.15244,  (2023).

\bibitem{universal7}
{\sc Y.~Yang and J.~He}, {\em Deeper or wider: A perspective from optimal generalization error with sobolev loss}, Proceedings of the 41st International Conference on Machine Learning, PMLR 235:56109-56138,  (2024).

\bibitem{structure7}
{\sc Y.~Yang, Y.~Wu, H.~Yang, and Y.~Xiang}, {\em Nearly optimal approximation rates for deep super {R}e{L}{U} networks on {S}obolev spaces}, arXiv preprint arXiv:2310.10766,  (2023).

\bibitem{universal5}
{\sc Y.~Yang, H.~Yang, and Y.~Xiang}, {\em Nearly optimal {V}{C}-dimension and pseudo-dimension bounds for deep neural network derivatives}, Advances in Neural Information Processing Systems, 36 (2023), pp.~21721--21756.

\bibitem{DeepRitz}
{\sc B.~Yu et~al.}, {\em The deep ritz method: a deep learning-based numerical algorithm for solving variational problems}, Communications in Mathematics and Statistics, 6 (2018), pp.~1--12.

\bibitem{wan}
{\sc Y.~Zang, G.~Bao, X.~Ye, and H.~Zhou}, {\em Weak adversarial networks for high-dimensional partial differential equations}, Journal of Computational Physics, 411 (2020), p.~109409.

\bibitem{ai4sci1}
{\sc L.~Zhang, J.~Han, H.~Wang, R.~Car, and W.~E}, {\em Deep potential molecular dynamics: a scalable model with the accuracy of quantum mechanics}, Physical review letters, 120 (2018), p.~143001.

\bibitem{train1}
{\sc H.~Zheng, Y.~Huang, Z.~Huang, W.~Hao, and G.~Lin}, {\em {H}om{P}{I}{N}{N}s: {H}omotopy physics-informed neural networks for solving the inverse problems of nonlinear differential equations with multiple solutions}, Journal of Computational Physics, 500 (2024), p.~112751.

\end{thebibliography}

\end{document}